\documentclass[12pt]{article}
\usepackage{amsmath,amssymb,amsfonts,amsthm}
\usepackage{eucal,oldgerm}
\usepackage[dvips]{graphics}
\usepackage{pstricks,pst-plot,pst-node,pst-coil}
\newpsobject{showgrid}{psgrid}{subgriddiv=1,griddots=5,gridlabels=0pt}
\usepackage[arrow,matrix,tips,2cell]{xy}

\newcommand{\com}{{\mathbb C}}

\newcommand{\PP}{{\mathbf{P}}}
\newcommand{\proj}{\PP}

\newcommand{\oh}{{\mathcal O}}

\newcommand{\MM}{{\mathcal{M}}}

\newcommand{\Q}{\mathbb{Q}}
\newcommand{\Z}{\mathbb{Z}}

\newcommand{\rarr}{\rightarrow}

\DeclareMathOperator{\Aut}{Aut}

\DeclareMathOperator{\FZ}{\mathsf{FZ}}
\DeclareMathOperator{\SQ}{\mathsf{SQ}}

\newtheorem{Theorem}{Theorem}
\newtheorem{Lemma}{Lemma}

\newtheorem{Proposition}[Lemma]{Proposition}
\newtheorem{Conjecture}{Conjecture}

\begin{document}
\title{Relations in the tautological ring of the
moduli space of curves}
\author{R. Pandharipande and A. Pixton}
\date{January 2020}
\maketitle

\vspace{-10pt}
{\centering \em Dedicated to David Mumford on the occasion of his 80th birthday\par}

\begin{abstract} 
The virtual geometry of the moduli space of stable
quotients is used to obtain Chow relations among the $\kappa$
classes on the moduli space of nonsingular genus $g$
curves. In a series of steps, the stable quotient
relations are rewritten in successively  simpler 
forms. The final result is the proof of the Faber-Zagier
relations (conjectured in 2000).
\end{abstract}

\maketitle

\setcounter{tocdepth}{1} 
\tableofcontents

%%%%%%%%%%%%%%%%%%%%%%%%%%%%%%%%%%%%%%%%%%%%%%%%%%%%%%%%%%%%%%%%%%%%%%%%%%%

\setcounter{section}{-1}
\section{Introduction}

\subsection{Tautological classes}
For $g\geq 2$, let $\MM_g$ be the moduli space of nonsingular, projective,
genus $g$ curves over $\com$, and let
\begin{equation}\label{g55}
\pi: \mathcal{C}_g \rightarrow \mathcal{M}_g 
\end{equation}
be the universal curve. We view $\MM_g$ and $\mathcal{C}_g$ as
nonsingular, quasi-projective, Deligne-Mumford stacks. 
However, the orbifold
perspective is sufficient for most of our purposes.

The relative dualizing sheaf $\omega_\pi$ of the morphism \eqref{g55}
is used to define  
the cotangent line class
$$\psi = c_1(\omega_\pi)\in A^1(\mathcal{C}_g,\mathbb{Q})\ .$$
The $\kappa$ classes are defined by push-forward,
$$\kappa_r = \pi_*( \psi^{r+1}) \in A^{r}(\MM_g)\ .$$
The {\em tautological ring} 
$$R^*(\MM_g) \subset A^*(\MM_g, \mathbb{Q})$$
is the $\mathbb{Q}$-subalgebra generated by all of the
$\kappa$ classes. 
Since 
$$\kappa_{0}= 2g-2 \in \mathbb{Q}$$
is a multiple of the fundamental class, we need not take 
$\kappa_0$ as a generator.
There is a canonical quotient
$$\mathbb{Q}[\kappa_1,\kappa_2, \kappa_3, \ldots] 
\stackrel{q}{\longrightarrow} R^*(\MM_g) \longrightarrow 0\ .$$
We study here the ideal of relations among the
$\kappa$ classes, the kernel of $q$.

We may also define a tautological ring $RH^*(\MM_g) \subset 
H^*(\MM_g,\mathbb{Q})$
generated by the $\kappa$ classes in cohomology.
Since there is a natural factoring
$$\mathbb{Q}[\kappa_1,\kappa_2, \kappa_3, \ldots] 
\stackrel{q}{\longrightarrow} R^*(\MM_g) 
\stackrel{c}{\longrightarrow} RH^*(\MM_g)$$
via the cycle class map $c$, algebraic relations among the
$\kappa$ classes are also cohomological relations.
Whether or not there exist {\em more} cohomological relations is not yet
settled.

There are two basic motivations for the study of the 
tautological rings $R^*(\MM_g)$.
The first is Mumford's conjecture, proven in 2002 
by Madsen and  Weiss \cite{MW},
$$\lim_{g\rightarrow \infty} H^*(\MM_g, \mathbb{Q}) \ = \mathbb{Q}[\kappa_1,
\kappa_2, \kappa_3, \ldots ]
,$$
determining the {\em stable} cohomology of the moduli 
of curves. While the $\kappa$ classes do not exhaust
$H^*(\MM_g,\mathbb{Q})$, there are no other stable classes.
The study of $R^*(\MM_g)$ undertaken 
here is from the opposite perspective ---
we are interested in the ring of $\kappa$ classes for fixed $g$.

The second motivation is from a large body of cycle
class calculations on $\MM_g$ (often related to Brill-Noether
theory). The answers invariably lie in the tautological ring
$R^*(\MM_g)$. The first definition of the tautological rings
by Mumford \cite{M} was at least partially motivated by such
algebro-geometric cycle constructions.

\subsection{Faber-Zagier conjecture}
Faber and  Zagier have conjectured a remarkable set
of relations among the $\kappa$ classes in $R^*(\MM_g)$.
Our main result is a proof of the Faber-Zagier relations, stated
as Theorem 1 below, 
by a geometric construction involving the virtual class
of the moduli space of stable quotients.

To write the Faber-Zagier relations, we will require
the following notation.
Let the variable set
$$\mathbf{p} = \{\ p_1,p_3,p_4,p_6,p_7,p_9,p_{10}, \ldots\ \}$$
be indexed by positive integers {\em not} congruent
to $2$ modulo $3$.
Define the series
\begin{multline*}
\Psi(t,\mathbf{p}) =
(1+tp_3+t^2p_6+t^3p_9+\ldots) \sum_{i=0}^\infty \frac{(6i)!}{(3i)!(2i)!} t^i
\\ +(p_1+tp_4+t^2p_7+\ldots) 
\sum_{i=0}^\infty \frac{(6i)!}{(3i)!(2i)!} \frac{6i+1}{6i-1} t^i \ .
\end{multline*}
Since $\Psi$ has constant term 1, we may take the logarithm.
Define the constants $C_r^{\text{\tiny{{\sf FZ}}}}(\sigma)$ by the formula
$$\log(\Psi)= 
\sum_{\sigma}
\sum_{r=0}^\infty C_r^{\text{\tiny{{\sf FZ}}}}(\sigma)\ t^r 
\mathbf{p}^\sigma
\ . $$
The above sum is over all partitions $\sigma$ of size 
$|\sigma|$ 
which avoid 
 parts congruent to 2 modulo 3. The empty partition is included
in the sum.
To the partition  $\sigma=1^{n_1}3^{n_3}4^{n_4} \cdots$, we associate
the monomial
$\mathbf{p}^\sigma= p_1^{n_1}p_3^{n_3}p_4^{n_4}\cdots$.
Let 
$$\gamma^{\text{\tiny{{\sf FZ}}}}
= 
\sum_{\sigma}
 \sum_{r=0}^\infty C_r^{\text{\tiny{{\sf FZ}}}}(\sigma)
\ \kappa_r t^r 
\mathbf{p}^\sigma
\ .
$$
For a series $\Theta\in \mathbb{Q}[\kappa][[t,\mathbf{p}]]$ in the variables $
\kappa_i$, $t$, and $p_j$, let
$[\Theta]_{t^r \mathbf{p}^\sigma}$ denote the
 coefficient of $t^r\mathbf{p}^\sigma$
(which is a polynomial in the $\kappa_i$).

\begin{Theorem} \label{dddd}
{ In $R^r(\MM_g)$, the Faber-Zagier relation
$$
\big[ \exp(-\gamma^{\text{\tiny{{\sf FZ}}}}) \big]_{t^r \mathbf{p}^\sigma}  = 0$$
holds when
$g-1+|\sigma|< 3r$ and
$g\equiv r+|\sigma|+1 \mod 2$.}
\end{Theorem}

The dependence upon the genus $g$ in the Faber-Zagier relations of 
Theorem \ref{dddd} occurs in the inequality, the modulo 2 
restriction, and via $\kappa_0=2g-2$. 
For a given genus $g$ and codimension $r$,
Theorem \ref{dddd} provides only {\em finitely} many relations.
While not immediately clear
from the definition,
the $\mathbb{Q}$-linear span of the
Faber-Zagier relations determines an ideal in $\mathbb{Q}[\kappa_1,\kappa_2, \kappa_3, \ldots]$ --- the matter is discussed in  Section \ref{pppp} and a subset of the Faber-Zagier relations generating the same ideal is described.

As a corollary of our proof of Theorem \ref{dddd} via the moduli space of stable
quotients, we obtain the following stronger boundary result.
If $g-1+|\sigma|< 3r$ and
$g\equiv r+|\sigma|+1 \mod 2$, then
\begin{equation}\label{bbbbb}
\big[ \exp(-\gamma^{\text{\tiny{{\sf FZ}}}}) \big]_{t^r \mathbf{p}^\sigma}  \in R^*(\partial\overline{\mathcal{M}}_g)\ 
.
\end{equation}
Not only is the Faber-Zagier relation 0 on $R^*(\mathcal{M}_g)$, but the
relation is equal to a tautological class on the boundary of
the moduli space $\overline{\mathcal{M}}_g$.
%A precise conjecture for 
%the boundary terms has been proposed in \cite{Pix}.

\subsection{Gorenstein rings}
By results of  Faber \cite{Faber} and  Looijenga \cite{L}, we have
\begin{equation}\label{gvvg}
\text{dim}_\mathbb{Q}\ R^{g-2}(\MM_g) =1, \ \ \ R^{>g-2}(\MM_g) =0 .\ 
\end{equation}
A canonical parameterization of $R^{g-2}(\MM_g)$ is obtained via 
integration. Let 
$$\mathbb{E} \rightarrow \MM_g$$
be the {\em Hodge bundle} with fiber $H^0(C,\omega_C)$ over the
moduli point $[C]\in \MM_g$. Let $\lambda_k$ denote the
$k^{th}$ Chern class of $\mathbb{E}$.
The linear map
$$\epsilon: \mathbb{Q}[\kappa_1,\kappa_2, \kappa_3, \ldots] \longrightarrow
\mathbb{Q}, \ \ \ \ \ \ \ 
f(\kappa) \stackrel{\epsilon}{\longmapsto} \int_{\overline{\MM}_g} f(\kappa) \cdot \lambda_g\lambda_{g-1}
$$
factors through $R^*(\MM_g)$ and
determines an isomorphism
$$\epsilon: R^{g-2}(\MM_g) \cong \mathbb{Q}$$
via the non-trivial evaluation
\begin{equation} 
\label{sdffds}
\int_{\overline{\MM}_{g}} \kappa_{g-2} \lambda_g \lambda_{g-1}
= \frac{1}{2^{2g-1}(2g-1)!!} \frac{|B_{2g}|}{2g} \ .
\end{equation}
A survey of the construction and properties of $\epsilon$ 
can be found in \cite{FPlog}.

The evaluations under $\epsilon$ of all polynomials in the
$\kappa$ classes are determined by the following formulas.
First,  
the Virasoro constraints for surfaces \cite{GP} imply a related
evaluation
previously conjectured in \cite{Faber}: 
\begin{equation} \label{lamgg}
\int_{\overline{\MM}_{g,n}} 
\psi_1^{\alpha_1} \cdots \psi_n^{\alpha_n} \lambda_g \lambda_{g-1}
= \frac{(2g+n-3)! (2g-1)!!}{(2g-1)!\prod_{i=1}^n (2\alpha_i-1)!!}
\int_{\overline{\MM}_{g}} \kappa_{g-2} \lambda_g \lambda_{g-1},
\end{equation}
where $\alpha_i>0$.
Second, a
basic relation (due to Faber) holds:
\begin{equation}\label{ffbb}
\int_{\overline{\MM}_{g,n}} 
\psi_1^{\alpha_1} \cdots \psi_n^{\alpha_n} \lambda_g \lambda_{g-1} =
\sum_{\sigma\in \mathsf{S}_n} \int_{\overline{\MM}_g}
\kappa_\sigma \lambda_g\lambda_{g-1}\ .
\end{equation}
The sum on the right is over all elements of the
symmetric group $\mathbb{S}_n$,
$$\kappa_\sigma = \kappa_{|c_1|} \ldots \kappa_{|c_r|}$$
where $c_1,\ldots ,c_r$ is the set partition obtained
from the cycle decomposition of $\sigma$, and
$$|c_i|= \sum_{j\in c_i} (\alpha_j -1)\ .$$
Relation \eqref{ffbb} is triangular and can be inverted
to express the $\epsilon$ evaluations of the $\kappa$
monomials in terms of \eqref{lamgg}.

Computations of the tautological rings in low genera led
  Faber to formulate the following conjecture in  1991.

\begin{Conjecture} For all $g\geq 2$ and all $0 \leq k \leq g-2$, the pairing
\begin{equation} \label{PPP}
R^{k}(\MM_g) \times R^{g-2-k}(\MM_g) \xrightarrow{\ \ \epsilon \ \circ\ 
\cup\ \ } \mathbb{Q}\end{equation}
is perfect.
\end{Conjecture}

\noindent The pairing \eqref{PPP} is the 
ring multiplication $\cup$
of $R^*(\MM_g)$ composed with $\epsilon$.
A perfect pairing identifies the first vector space with the
dual of the second. If Faber's conjecture is true in genus $g$, then
$R^*(\MM_g)$ is a Gorenstein local ring.

Let $\mathcal{I}_g \subset R^*(\MM_g)$ be the ideal determined by the
kernel of the pairing \eqref{PPP} in Faber's conjecture.
Define the
{\em Gorenstein quotient}
$$R^*_{\mathrm{G}}(\MM_g) = \frac{R^*(\MM_g)}{\mathcal{I}_g}\ .$$ 
If Faber's conjecture is true for $g$, then $\mathcal{I}_g=0$
and $R^*_{\mathrm{G}}(\MM_g)= R^*(\MM_g)$.

The pairing \eqref{PPP} can be evaluated directly on
polynomials in the $\kappa$ classes via \eqref{sdffds}-\eqref{ffbb}.
The Gorenstein quotient $R_{\mathrm{G}}^*(\MM_g)$ is
completely determined by the $\kappa$ evaluations and
the ranks \eqref{gvvg}. The ring $R_{\mathrm{G}}^*(\MM_g)$ can therefore
be
studied  as a purely algebro-combinatorial object.

 Faber and  Zagier conjectured the relations of Theorem \ref{dddd}
from a concentrated study of the
Gorenstein quotient $R^*_{\mathrm{G}}(\MM_g)$. The Faber-Zagier
relations were first written in 2000 and were
proven to hold in $R^*_{\mathrm{G}}(\MM_g)$ in 2002. 
The validity of the Faber-Zagier relations in $R^*(\MM_g)$ has been
an open question since then.

\subsection{Other relations?}

By substantial computation,  Faber has verified 
Conjecture 1 holds for genus $g< 24$. Moreover, his calculations
show
the Faber-Zagier set yields {\em all} relations among $\kappa$
classes in $R^*(\MM_g)$ for $g< 24$.
However, he finds
 the Faber-Zagier relations of
Theorem \ref{dddd} do {\em not} yield a Gorenstein quotient in genus 24. 
Let $$\mathsf{FZ}_g \subset \mathbb{Q}[\kappa_1, \kappa_2, \kappa_3, \ldots]$$
be the ideal determined by the Faber-Zagier relations of Theorem \ref{dddd},
and let
$$R^*_{\mathrm{FZ}}(\MM_g) = \frac{\mathbb{Q}[\kappa_1, \kappa_2, \kappa_3, \ldots]}{\mathsf{FZ}_g}\ .$$
 Faber finds a mismatch in codimension 12,
\begin{equation}\label{ineq}
R^{12}_{\mathrm{FZ}}(\MM_{24}) \neq R^{12}_{\mathrm{G}}(\MM_{24})\ .
\end{equation}
Exactly 1 more relation holds in the Gorenstein
quotient.

To the best of our knowledge, a relation in $R^*(\MM_g)$ which is
not in the span of the Faber-Zagier
relations of Theorem \ref{dddd} has not yet been found.
The following prediction is consistent with all present calculations.

\begin{Conjecture} For all $g\geq 2$, the kernel of 
$$\mathbb{Q}[\kappa_1,\kappa_2, \kappa_3, \ldots] 
\stackrel{q}{\longrightarrow} R^*(\MM_g) \longrightarrow 0$$
is the Faber-Zagier ideal $\mathsf{FZ}_g$.
\end{Conjecture}

\noindent 
Conjectures 1 and 2 are both true for $g<24$.
By the inequality \eqref{ineq}, Conjectures 1 and 2 can {\em not}
both be true for all $g$. Which is false? 

Finally, we note the above discussion might have a different
outcome if the tautological ring $RH^*(\MM_g)$ in
cohomology is considered instead. Perhaps there are more
relations in cohomology? These questions provide
a very interesting line of inquiry.

%Theorem 5 is much more efficient than Theorem 1 for several
%reasons. Theorem 5 only provides finitely many relations
%in $R^r(\MM_g)$ for fixed $g$ and $r$, and thus may be 
%calculated completely. 
%When the relations yield a Gorenstein
%ring with socle in $R^{g-2}(\MM_g)$, no further relations are possible.
%However, the relations of Theorem 5 do not always yield such a
%Gorenstein ring (failing first in genus 24 as checked by Faber).
%For $g<24$, Faber's calculations show
%Theorem 5 does provide all relations in $R^*(\MM_g)$.
%For higher genus $g\geq 24$, either Theorem 5 fails to
%provide all the relations in $R^*(\MM_g)$ {\em or}
%$R^*(\MM_g)$ is not Gorenstein. 

\subsection{Plan of the paper}

We start the paper in Section \ref{FFF} with a modern treatment of Faber's 
classical
construction of relations among the
$\kappa$ classes. The result, in Wick form, is stated
as Theorem \ref{ttt} of Section \ref{wickf}.
While the outcome is an effective source of
relations, their complexity has so far defied a complete 
analysis.

After reviewing stable quotients on curves in
Section \ref{stq},
we derive an explicit set of $\kappa$ relations
from the virtual geometry of the moduli space of stable quotients
in Section \ref{MOP}.
The resulting equations are more tractable than those
obtained by classical methods.
In a series of steps, the stable quotient
relations are transformed to simpler and simpler forms.
The first step, Theorem  \ref{mmnn}, comes almost immediately from
the virtual localization formula \cite{GrP} applied to 
the moduli space of stable quotients.
After further analysis in Section \ref{LLL}, the simpler form of 
Proposition \ref{better} 
is found.  A change of variables is applied in Section \ref{trans}
that transforms the relations to Proposition \ref{best}.
Our final result, Theorem \ref{dddd}, establishes
the previously conjectural set of tautological
relations proposed in 2000  by  Faber and  Zagier. 
The proof of Theorem \ref{dddd} is completed in Section \ref{pppp}.

\subsection{History and subsequent results}

We first presented our proof of the Faber-Zagier relations in 
a series of lectures at Humboldt University
in Berlin during the conference {\em Intersection theory on moduli space}
in 2010. We thank G. Farkas for the invitation to
speak there.
A detailed set of lecture notes, which is
the origin of
the current paper,
is available \cite{PP}. The current paper itself was written in 2012.

Ten years after
our lectures in Berlin, relations among the $\kappa$ classes which are
not in the Faber-Zagier set have  not  been found despite serious efforts.
Conjecture 2 stating that
{\em the Faber-Zagier set spans all the relations in $R^*(M_g)$}
appears now to be more widely believed.
In particular, since the Faber-Zagier relations fail to define
a Gorenstein ring for $g=24$, Conjecture 1 is
expected to be false.

A natural question is whether Theorem \ref{dddd} can be extended
to yield explicit relations in the tautological ring of
the moduli spaces  $\overline{\MM}_{g,n}$ of stable curves.
A precise conjecture of exactly such an extension was given by Pixton
\cite{Pix} in 2012 which provides
explicit boundary expressions for \eqref{bbbbb}. In fact,
 Pixton \cite{Pix} has conjectured
a complete set $\mathsf{P}$ of relations in the Chow
ring $A^*(\overline{M}_{g,n})$ among all tautological classes on the moduli spaces of stable curves.
These include not only  the $\kappa$ classes
from the interior, but also the cotangent classes
of the $n$ markings {\em and} the boundary classes. In 2015,
Pixton's conjectured
relations were proven in \cite{PPZ}
to hold in $H^*(\overline{M}_{g,n})$ using  Witten's virtual
fundamental class of the moduli 
of $3$-spin curves and the $R$-matrix of the  Givental-Teleman classification
of CohFTs. In 2017, Janda  found a second approach \cite{Jan}
using the virtual fundamental class of the space of stable
maps to $\mathbb{P}^1$ which proves the set
$\mathsf{P}$ holds in Chow. For an introduction to these developments,
we refer the reader to \cite{Calc, Coh}.

The papers  \cite{Jan,PPZ}
also provide new proofs of the Faber-Zagier relations on $\MM_g$.
The proof of the Faber-Zagier relations presented here is geometrically
simpler (and avoids the Givental-Teleman classifications of CohFTs), but
is combinatorially more complicated than the later methods.

Pixton's relations capture all the previously known special relations
in the moduli spaces: WDVV in genus 0, Getzler's relation \cite{Get} in genus 1, the
 Belorousski-Pandharipande relation  \cite{BP} in genus 2, and so on. 
In fact,  no relation outside the set $\mathsf{P}$ has ever been found.
Whether $\mathsf{P}$ is the complete set of relations
among tautological classes is now the central open question
in the study the tautological ring.

In the past decade, the Faber-Zagier relations have been
connected to several different approaches to the geometry
of the moduli space of curves \cite{BJP, CY,Jan, Jan2, PPZ, PPZ2, RW,Y}. However, the classical
approach to tautological relations on $\MM_g$ given by Faber's original
construction has still not been proven to yield exactly the Faber-Zagier
set (a result strongly suggested by computer calculations).
The Wick formulation in Section 1 of Faber's construction remains
the most likely path to a complete analysis of the classical relations.

\subsection{Acknowledgements}
Discussions with C. Faber
played an important role in the development of our ideas. 
%The study of descendents
%for 3-fold sheaf theories in \cite{MNOP2,pt2} 
%motivated several 
%aspects of the paper. 
Much of the research reported here was done during visits of A.P.
to IST Lisbon during  the year 2010-11. 
R.P. was supported in Lisbon by a Marie Curie fellowship
and a grant from the Gulbenkian foundation.

We thank the {\em Forschungsinstitut f\"ur Mathematik} at ETHZ
for hosting several visit of A.P. in Z\"urich.
 R.P. was partially supported by 
  SNF-200021143274, SNF-200020182181, ERC-2017-786580-MACI,
 SwissMAP,
 and the Einstein Stiftung.
A.P. was supported by an NDSEG graduate fellowship and a
fellowship from the Clay Mathematics Institute.

The project has received funding from the European Reseach Council (ERC)
under the European Union Horizon 2020 Research and Innovation
Program (grant No. 786580).

\section{Classical vanishing relations}
\label{FFF}

\subsection{Construction}
Faber's original relations in his article
{\em Conjectural description of the tautological ring} \cite{Faber}
are obtained from a very simple geometric construction.
As before, let 
$$\pi: \mathcal{C}_g \rightarrow \mathcal{M}_g$$
be the universal curve over the moduli space,
and let 
$$\pi^d: \mathcal{C}^d_g \rightarrow \mathcal{M}_g $$
be the map associated to the $d^{th}$ 
fiber product of the universal curve.
For every point $[C,p_1,\ldots,p_d]\in\mathcal{C}^d_g$,
we have the restriction map
\begin{equation}\label{gy77}
H^0(C,\omega_C) \rightarrow H^0(C,\omega_C|_{p_1+\ldots+p_d})\ .
\end{equation}
Since the canonical bundle $\omega_C$ has degree $2g-2$,
the restriction map is injective if $d>2g-2$.
Let $$\Omega_d \rightarrow \mathcal{C}_g^d$$ be the
rank $d$ bundle with fiber
$H^0(C,\omega_C|_{p_1+\ldots+p_d})$ over the moduli point
$[C,p_1,\ldots,p_d]\in\mathcal{C}^d_g$.
If $d>2g-2$,
the restriction map \eqref{gy77} yields 
an exact sequence over $\mathcal{C}^d$, 
$$ 0 \rightarrow \mathbb{E} \rightarrow \Omega_{d}
\rightarrow Q_{d-g} \rightarrow 0$$
where $\mathbb{E}$ is the rank $g$ Hodge bundle and
$Q_{d-g}$ is the quotient
bundle of rank $d-g$.
We see
$$c_k(Q_{d-g}) = 0\in A^k(\mathcal{C}^d_g)\ \ \ \text{for} \ \ \ k>d-g \ .$$
After cutting the vanishing Chern classes $c_k(Q_{d-g})$ down with
cotangent line and diagonal classes
in $\mathcal{C}^d_g$ and 
pushing-forward via $\pi^d_*$ to $\MM_g$,
we arrive at Faber's relations in $R^*(\MM_g)$.

\subsection{Wick form} \label{wickf}
From our point of view,
at the center of 
Faber's relations in
\cite{Faber}
is the function
$$\Theta(t,x) = \sum_{d=0}^\infty \prod_{i=1}^d {(1+it)} \ 
\frac {(-1)^d}{d!} \frac{x^d}{t^{d}} \ .$$
The differential equation
$$ t(x+1)\frac{d}{dx} \Theta + (t+1) \Theta  = 0 \  $$
is easily found.
Hence, we obtain the following result.

\begin{Lemma}
$\Theta = (1+x)^{-\frac{t+1}{t}}\ .$
\end{Lemma}

We introduce a variable set $\mathbf{z}$ indexed
by pairs of integers
$$\mathbf{z} = \{ \ {z}_{i,j} \ | \ i \geq 1, \ \ j\geq i-1 \  \} \ .$$
For monomials
$$\mathbf{z}^\sigma = \prod_{i,j} z_{i,j}^{\sigma_{i,j}},$$
 we define
$$\ell(\sigma) = \sum_{i,j} i \sigma_{i,j}, \ \ \
|\sigma| = \sum_{i,j} j \sigma_{i,j} \ .$$
Of course $|\text{Aut}(\sigma)| = \prod_{i,j} \sigma_{i,j} !$ \ .

The variables $\mathbf{z}$ are used to define
a differential operator
$$ \mathcal{D} = \sum_{i,j}  z_{i,j}\  t^j  \left( 
x\frac{d}{dx}\right) ^i\ .$$
After applying $\exp(\mathcal{D})$ to $\Theta$, we obtain
\begin{eqnarray*}
\Theta^{\mathcal{D}} & = &  \exp(\mathcal{D})\  \Theta \\
& = & 
\sum_\sigma \sum_{d=0}^\infty  
\prod_{i=1}^d {(1+it)} \ 
\frac {(-1)^d}{d!}  \frac{x^d}{t^{d}}\  
\frac{d^{\ell(\sigma)} t^{|\sigma|}
{\mathbf{z}}^{\sigma}}
{|\text{Aut}(\sigma)|}
\end{eqnarray*}
where $\sigma$ runs over all monomials in the
variables $\mathbf{z}$. Define
constants $C^d_r(\sigma)$ by the formula
$$\log(\Theta^{\mathcal{D}})= 
\sum_{\sigma}
\sum_{d=1}^\infty \sum_{r=-1}^\infty C^d_{r}(\sigma)\ t^r \frac{x^d}{d!}
\mathbf{z}^\sigma
\ .$$
By an elementary application of Wick's formula (as explained
in Section \ref{cc12} below), the
$t$ dependence of $\log(\Theta^{\mathcal{D}})$
has at most simple poles.

Finally, we consider the following function,
\begin{equation}
\gamma^{\text{\tiny{{\sf F}}}}=  \sum_{i\geq 1} \frac{B_{2i}}{2i(2i-1)} \kappa_{2i-1} t^{2i-1}
+ 
\sum_{\sigma}
\sum_{d=1}^\infty \sum_{r=-1}^\infty C^d_r(\sigma)
\ \kappa_r t^r \frac{x^d}{d!}
\mathbf{z}^\sigma
\ . \label{fafa}
\end{equation}
The Bernoulli numbers appear in the first term,
$$\sum_{k=0}^ \infty B_k \frac{u^k}{k!}= \frac{u}{e^u-1}\ .$$
Denote the $t^rx^d \mathbf{z}^\sigma$ coefficient of $\exp(-\gamma^{\text{\tiny{{\sf F}}}})$ by
$$\big[ \exp(-\gamma^{\text{\tiny{{\sf F}}}}) \big]_{t^rx^d \mathbf{z}^\sigma} 
\in \mathbb{Q}[\kappa_{-1},
\kappa_0,\kappa_1,
\kappa_2, \ldots] \ .$$
Our form of Faber's equations is the following result.

\begin{Theorem}\label{ttt}
In $R^r(\MM_g)$, the relation
$$
\big[ \exp(-\gamma^{\text{\tiny{{\sf F}}}}) \big]_{t^rx^d \mathbf{z}^\sigma}  = 0$$
holds when $r>-g+|\sigma|$ and $d>2g-2$.
\end{Theorem}

In the tautological ring $R^*(\MM_g)$, the standard conventions
$$\kappa_{-1}=0, \ \ \ \ \kappa_{0}=2g-2$$ are
 followed. 
For fixed $g$ and $r$, Theorem \ref{ttt}
provides infinitely many relations by increasing $d$.
The variables $z_{i,j}$ efficiently encode both
the cotangent and diagonal operations studied
in \cite{Faber}. In particular, the relations of Theorem \ref{ttt}
 are equivalent to a 
mixing of all cotangent and diagonal operations
studied there. The proof of Theorem \ref{ttt} is presented
in Section \ref{ppp}.

\vspace{10pt}
While Theorem \ref{ttt} has an appealingly simple geometric origin, 
the relations
do not seem to fit the other forms we will see later.
In particular, we do not know how to derive Theorem \ref{dddd}
from Theorem \ref{ttt}. Extensive computer calculations by Faber
suggest the following.

\begin{Conjecture}
For all $g\geq 2$, the relations of Theorem \ref{ttt} are
equivalent to the Faber-Zagier relations.
\end{Conjecture}

In particular, despite significant effort, the relation in
$R_{\mathrm{G}}^{12}(\MM_{24})$ which is missing in $R_{\mathsf{FZ}}^{12}(\MM_{24})$
has {\em not} been found via Theorem \ref{ttt}. Other geometric
strategies have so far also failed to find 
the missing relation \cite{RW,Y}.

\subsection{Proof of Theorem \ref{ttt}}
\label{ppp}

\subsubsection{The Chern roots of $\Omega_d$}

Let $\psi_i \in A^1(\mathcal{C}^d_g, \mathbb{Q})$
be the first Chern class of the
relative dualizing sheaf $\omega_\pi$ pulled back from the $i^{th}$ factor,
$$\mathcal{C}^d_g \rightarrow \mathcal{C}_g\ .$$
For $i\neq j$, let $D_{ij} \in A^1(\mathcal{C}^d_g, \mathbb{Q})$ be the 
class of the diagonal $\mathcal{C}_g \subset \mathcal{C}^2_g$
pulled-back from the product of the
$i^{th}$ and $j^{th}$ factors,
$$\mathcal{C}^d_g \rightarrow \mathcal{C}^2_g\ .$$
Finally, let
$$\Delta_i = D_{1,i} + \ldots + D_{i-1,i}\ \in A^1(\mathcal{C}_g^d,\mathbb{Q})\ $$
following the convention $\Delta_1=0$.
The Chern roots of $\Omega_d$,
\begin{eqnarray}
c_t(\Omega_d) & = & \prod_{i=1}^d 1+(\psi_i - \Delta_i)t \label{fredd}
\\
& = & \nonumber
(1+\psi_1t) \cdot \big(1+(\psi_2-D_{12})t\big) \cdots
 \left(1+\Big(\psi_d -\sum_{i=1}^{d-1}D_{id}\Big)t\right)
\end{eqnarray}
are obtained by a simple induction, see \cite{Faber}.

We may expand the right side of \eqref{fredd} fully. 
The resulting expression is a  polynomial in the 
$d+ \binom{d}{2}$ variables. 
$${\psi}_1,\ldots, {\psi}_d, -D_{12},-D_{13}, 
\ldots,- D_{d-1,d}\ .$$ 
The
sign on the diagonal variables is chosen
because of the self-intersection formula
$$(-D_{ij})^2= \psi_i(-D_{ij})=\psi_j(-D_{ij})\ . $$
Let $M_r^d$ denote the coefficient in degree $r$,
$$c_t( \Omega_d) =\sum_{r=0}^\infty
M_r^d({\psi}_i,- D_{ij}) \ t^r.$$

\begin{Lemma} \label{gcd2} After setting all the variables to 1,
$$\sum_{r=0}^\infty M_r^d({\psi}_i=1,-D_{ij}=1) \ t^r
\ = \ \prod_{i=1}^d (1+it).$$
\end{Lemma}

\begin{proof}
The results follows immediately from the Chern roots \eqref{fredd}.
\end{proof}

Lemma \ref{gcd2} may be viewed  counting the number of
terms in the expansion of the total Chern class $c_t(\Omega_d)$.

\subsubsection{Connected counts}\label{cc12}
A monomial in the diagonal variables
\begin{equation} \label{gqq6}
D_{12},D_{13}, \ldots,D_{d-1,d}
\end{equation}
determines a set partition of $\{1, \ldots, d\}$
by the diagonal associations.
For example, the monomial $3D_{12}^2D_{1,3} D_{56}^3$ determines
the set partition
$$\{1,2,3\} \ \cup \ \{4\}\ \cup \ \{5,6\}$$
in the $d=6$ case.
A monomial in the variables \eqref{gqq6} is
{\em connected} if the corresponding
set partition consists of a single part with $d$ elements.

A monomial in the variables
\begin{equation}\label{gffvv}
{\psi}_1,\ldots, {\psi}_d, -D_{12},-D_{13}, 
\ldots, -D_{d-1,d}\ 
\end{equation} 
is connected if
the corresponding monomial in the diagonal variables
obtained by setting all ${\psi}_i=1$
is connected.
Let $S^d_r$ be the summand of the evaluation  $M^d_r({\psi}_i=1, -D_{ij}=1)$ 
consisting of the contributions of
only the connected monomials.

\begin{Lemma} \label{llgg}
We have
$$\sum_{d=1}^\infty \sum_{r=0}^d
S_r^d\  t^r \frac{x^d}{d!}  =
\log\left( 1+\sum_{d=1}^\infty \prod_{i=1}^d (1+it)
\frac{x^d}{d!}
\right)\ .
$$
\end{Lemma}

\begin{proof}
By a standard application of Wick's formula, the connected
and disconnected counts are related by exponentiation,
$$\exp\left(\sum_{d=1}^\infty \sum_{r=0}^d
S_r^d \ t^r\frac{x^d}{d!}\right) =
1+ \sum_{d=1}^\infty \sum_{r=0}^\infty
M_r^d(\psi_i=1, -D_{ij}=1) \ t^r\frac{x^d}{d!} \ .$$
The right side is then evaluated by Lemma \ref{gcd2}.
\end{proof}

Since a connected monomial in the variables \eqref{gffvv}
must have at least $d-1$ factors of the variables $-D_{ij}$,
we see $S^d_r =0$ if $r<d-1$. Using the self-intersection
formulas, we obtain
\begin{equation}
\sum_{d=1}^ \infty \sum_{r=0}^d \pi^d_*\big(c_r(\Omega_d)\big)\ t^r\frac{x^d}{d!}
= \exp\left(\sum_{d=1}^\infty \sum_{r=0}^d
S_r^d (-1)^{d-1}\kappa_{r-d}\ t^r\frac{x^d}{d!}\right) \ . 
\end{equation}
To account for the alternating factor $(-1)^{d-1}$ and the $\kappa$ subscript,
we define the coefficients $C^d_r$ by 
$$\sum_{d=1}^\infty \sum_{r\geq -1}^d
C_r^d\  t^r \frac{x^d}{d!}  =
\log\left( 1+\sum_{d=1}^\infty \prod_{i=1}^d (1+it)
\frac{(-1)^d}{t^d} \frac{x^d}{d!}
\right)\ .
$$
The vanishing  $S^d_{r<d-1}=0$ implies the vanishing $C^d_{r<-1}=0$. 

The formula for the total Chern class of the Hodge bundle
$\mathbb{E}$ on $\mathcal{M}_g$ follows immediately from
Mumford's Grothendieck-Riemann-Roch calculation \cite{M},
$$c_t(\mathbb{E}) = \sum_{i\geq 1} \frac{B_{2i}}{2i(2i-1)} \kappa_{2i-1} t^{2i-1}\ .$$
Putting the above results together yields the following 
formula:
\begin{multline*}
\sum_{d=1}^ \infty \sum_{r\geq 0} \pi^d_*\big(c_r(Q_{d-g})\big)\ t^{r-d}\frac{x^d}{d!}
= \\
\exp\left(-\sum_{i\geq 1} \frac{B_{2i}}{2i(2i-1)} \kappa_{2i-1} t^{2i-1}
-\sum_{d=1}^\infty \sum_{r\geq -1}
C_r^d\ \kappa_r t^r \frac{x^d}{d!}
\right) \ .
\end{multline*}

\subsubsection{Cutting}
For $d>2g-2$ and $r>d-g$, we have the vanishing 
$$c_r(Q_{d-g})=0 \in 
A^r(\mathcal{C}^d_g, \mathbb{Q})\ .$$
Before pushing-forward via $\pi^d$, we will cut $c_r(Q_{d-g})$
with products of classes in $A^*(\mathcal{C}^d_g,\mathbb{Q})$.
With the correct choice of cutting classes, we will
obtain the relations of Theorem \ref{ttt}.

Let $(a,b)$ be a pair of integers satisfying $a\geq 0$ and $b\geq 1$.
We define the cutting class
\begin{equation}
\phi[a,b]= (-1)^{b-1}\sum_{|I|=b} \psi_I^a D_I   \label{kddc3}
\end{equation}
 where $I\subset \{1, \ldots, d\}$ is subset of
order $b$,
$D_I\in A^{b-1}(\mathcal{C}^d_g,\mathbb{Q})$
is the class of the corresponding small diagonal, and
$\psi_I$ is the cotangent line at the point indexed
by $I$. The class $\psi_I$ is well-defined on the
small diagonal indexed by $I$.
The degree of $\phi[a,b]$ is $a+b-1$.
The number of terms on the right side of \eqref{kddc3} is
a degree $b$ polynomial in $d$,
$$\binom{d}{b} = \frac{d^b}{b!} + \ldots +(-1)^{b-1}\frac{d}{b}\  $$
with no constant term.

The sign $(-1)^{b-1}$ in definition \eqref{kddc3}
is chosen to match the sign conventions of
the Wick analysis in Section \ref{cc12}.
For example,
$$\phi[0,2] = \sum_{i< j} (-D_{ij})\ , \ \ \ \
\phi[0,3]= \sum_{i<j<k} (-D_{ij})(-D_{jk}) .$$
The {\em number of terms} means the evaluation at
$\psi_I=1$ and $-D_{ij}=-1$.

A better choice of cutting class is obtained by the following
observation. For every pair of integers $(i,j)$ with
$i\geq 1$ and $j\geq i-1$, we can find a unique linear
combination
$$\Phi[i,j] = \sum_{a+b-1=j} \lambda_{a,b} \cdot \phi[a,b] , \ \ \ \lambda_{a,b}\in \mathbb{Q}$$
for which the evaluation of $\Phi[i,j]$ at $\psi_I=1$ and $-D_{ij}=-1$ is
$d^i$. By definition, $\Phi[i,j]$ is of pure degree $j$.

\subsubsection{Full Wick form}
We repeat the Wick analysis of Section \ref{cc12} for the 
Chern class of $Q_{d-g}$ cut by
the classes $\Phi[i,j]$ in order to write a formula
for
\begin{equation*}
\sum_{d=1}^ \infty \sum_{r\geq 0} \pi^d_*\left(
\exp\Big(\sum_{i,j} z_{i,j}t^j \Phi[i,j]  \Big) \cdot
c_r(Q_{d-g}) t^r\right)\ \frac{1}{t^d}\frac{x^d}{d!}
\end{equation*}
where the sum in the argument of the exponential
is over all $i\geq 1$ and $j\geq i-1$.
The variable set $\mathbf{z}$ introduced in Section \ref{wickf}
appears here. 
Since $\Phi[i,j]$ yields $d^i$ after evaluation at $\psi_I=1$ and
$-D_{ij}=-1$ and is of pure degree $j$, 
we conclude
\begin{equation}\label{mssx}
\sum_{d=1}^ \infty \sum_{r\geq 0} \pi^d_*\left(
\exp\Big(\sum_{i,j} z_{i,j}t^j \Phi[i,j]  \Big) \cdot
c_r(Q_{d-g}) t^r\right)\ \frac{1}{t^d}\frac{x^d}{d!}  = 
\exp(-\gamma^{\text{\tiny{{\sf F}}}})\ .
\end{equation}
%Here $\gamma$ is defined in \eqref{fafa}.

Let $d>2g-2$.
Since $c_s(Q_{d-g})=0$ for $s>d-g$, the $t^rx^d\mathbf{z}^\sigma$
coefficient of \eqref{mssx} vanishes if
$$r+d - |\sigma| > d-g$$
which is equivalent to $r> -g + |\sigma|$.
The proof of Theorem \ref{ttt} is complete. \qed

\section{Stable quotients} \label{stq}

\subsection{Stability} Our proof of the Faber-Zagier
relations in $R^*(M_{g})$ will be obtained
from the virtual geometry of the moduli space of
stable quotients.
We start by reviewing the basic definitions
and results of \cite{MOP}.

Let $C$ be a curve which is reduced and connected and has 
at worst nodal singularities. We require here only
unpointed curves. See \cite{MOP} for the definitions
in the pointed case.
Let $q$ be a quotient of the rank $N$ trivial bundle
 $C$,
\begin{equation*}
\com^N \otimes \oh_C \stackrel{q}{\rarr} Q \rarr 0.
\end{equation*}
If the quotient subsheaf $Q$ 
is locally free at the nodes and markings of $C$,
 then
$q$ is a {\em quasi-stable quotient}. 
 Quasi-stability of $q$ implies the associated
kernel,
\begin{equation*}
0 \rightarrow S \rightarrow
\com^N \otimes \oh_C \stackrel{q}{\rarr} Q \rarr 0,
\end{equation*}
is a locally free sheaf on $C$. Let $r$ 
denote the rank of $S$.

Let $C$ be a curve equipped
with a quasi-stable quotient $q$.
The data $(C,q)$ determine 
a {\em stable quotient} if
the $\mathbb{Q}$-line bundle 
\begin{equation}\label{aam}
\omega_C
\otimes (\wedge^{r} S^*)^{\otimes \epsilon}
\end{equation}
is ample 
on $C$ for every strictly positive $\epsilon\in \mathbb{Q}$.
Quotient stability implies
$2g-2 \geq 0$.

Viewed in concrete terms, no amount of positivity of
$S^*$ can stabilize a genus 0 component 
$$\proj^1\stackrel{\sim}{=}P \subset C$$
unless $P$ contains at least 2 nodes or markings.
If $P$ contains exactly 2 nodes or markings,
then $S^*$ {\em must} have positive degree.

A stable quotient $(C,q)$
yields a rational map from the underlying curve
$C$ to the Grassmannian $\mathbb{G}(r,N)$.
 We will
only require the ${\mathbb{G}}(1,2)=\proj^1$ case
for the proof Theorem \ref{dddd}.

\subsection{Isomorphism}
Let $C$ be a curve.
Two quasi-stable quotients
\begin{equation}\label{fpp22}
\com^N \otimes \oh_C \stackrel{q}{\rarr} Q \rarr 0,\ \ \
\com^N \otimes \oh_C \stackrel{q'}{\rarr} Q' \rarr 0
\end{equation}
on $C$ 
are {\em strongly isomorphic} if
the associated kernels 
$$S,S'\subset \com^N \otimes \oh_C$$
are equal.

An {\em isomorphism} of quasi-stable quotients
 $$\phi:(C,q)\rarr
(C',q')
$$ is
an isomorphism of curves
$$\phi: C \stackrel{\sim}{\rarr} C'$$
such that
the quotients $q$ and $\phi^*(q')$ 
are strongly isomorphic.
Quasi-stable quotients \eqref{fpp22} on the same
curve $C$
may be isomorphic without being strongly isomorphic.

The following result is proven in \cite{MOP} by
Quot scheme methods from the perspective
of geometry relative to a divisor.

\begin{Theorem} The moduli space of stable quotients 
$\overline{Q}_{g}({\mathbb{G}}(r,N),d)$ parameterizing the
data
$$(C,\  0\rarr S \rarr
\com^N\otimes \oh_C \stackrel{q}{\rarr} Q \rarr 0),$$
with {\em rank}$(S)=r$ and {\em deg}$(S)=-d$,
is a separated and proper Deligne-Mumford stack of finite type
over $\com$.
\end{Theorem}

\subsection{Structures}\label{strrr}
Over the moduli space of stable quotients, there is a universal
curve
\begin{equation}\label{ggtt}
\pi: U \rarr \overline{Q}_{g}({\mathbb{G}}(r,N),d)
\end{equation}
with a universal quotient
$$0 \rarr S_U \rarr \com^N \otimes \oh_U \stackrel{q_U}{\rarr} Q_U \rarr 0.$$
The subsheaf $S_U$ is locally 
free on $U$ because of the
stability condition.

The moduli space $\overline{Q}_{g}({\mathbb{G}}(r,N),d)$ is equipped
with two basic types of maps.
If $2g-2>0$, then the stabilization of $C$
determines a map
$$\nu:\overline{Q}_{g}({\mathbb{G}}(r,N),d) \rightarrow \overline{M}_{g}$$
by forgetting the quotient.

The general linear group $\mathbf{GL}_N(\com)$ acts on
$\overline{Q}_{g}({\mathbb{G}}(r,N),d)$ via 
the standard
action on $\com^N \otimes \oh_C$. The structures
$\pi$, $q_U$,
$\nu$ and the evaluations maps are all $\mathbf{GL}_N(\com)$-equivariant.

\subsection{Obstruction theory}
The moduli of stable
quotients 
 maps 
to the Artin stack of pointed domain curves
$$\nu^A:
\overline{Q}_{g}({\mathbb{G}}(r,N),d) \rightarrow {\mathcal{M}}_{g}.$$
The moduli  of stable quotients with fixed underlying
curve 
$[C] \in {\mathcal{M}}_{g}$
 is simply
an open set of the Quot scheme of $C$. 
The following result of \cite[Section 3.2]{MOP} is obtained from the
standard deformation theory of the Quot scheme.

\begin{Theorem}\label{htr}
The deformation theory of the Quot scheme 
determines a 2-term obstruction theory on the moduli space
$\overline{Q}_{g}({\mathbb{G}}(r,N),d)$ relative to
$\nu^A$
given by ${{RHom}}(S,Q)$.
\end{Theorem}

More concretely, for the stable quotient,
\begin{equation*}
0 \rightarrow S \rightarrow
\com^N \otimes \oh_C \stackrel{q}{\rarr} Q \rarr 0,
\end{equation*} the
deformation and obstruction spaces relative to $\nu^A$ are
$\text{Hom}(S,Q)$ and $\text{Ext}^1(S,Q)$
respectively. Since $S$ is locally free, the higher obstructions
$$\text{Ext}^{k}(S,Q)= H^{k}(C,S^*\otimes Q) = 0, \ \ \ k>1$$
vanish since $C$ is a curve.
An absolute 2-term obstruction theory on the moduli space
$\overline{Q}_{g}({\mathbb{G}}(r,N),d)$ is
obtained from Theorem \ref{htr} and the smoothness
of $\mathcal{M}_{g}$, see \cite{Beh,BF,GP}. The
 analogue of Theorem \ref{htr} for the Quot scheme of a {\it fixed} nonsingular
 curve was observed in \cite {MO}.

The $\mathbf{GL}_N(\com)$-action lifts to the
obstruction theory,
and the resulting virtual class is
defined in $\mathbf{GL}_N(\com)$-equivariant cycle theory,
$$[\overline{Q}_{g}({\mathbb{G}}(r,N),d)]^{vir} 
\in A_*^{\mathbf{GL}_N(\com)}
(\overline{Q}_{g}({\mathbb{G}}(r,N),d)).$$

For the construction of the Faber-Zagier relation,
 we are mainly interested in the open stable quotient
space 
$$\nu: Q_g(\mathbf{P}^1,d) \longrightarrow \MM_g$$
which is simply the corresponding relative Hilbert scheme.
However, 
we will require the full stable quotient space 
$\overline{Q}_g(\mathbf{P}^1,d)$ to prove the
Faber-Zagier relations can be completed over $\MM_g$
with tautological boundary terms.

\section{Stable quotients relations} \label{MOP}

\subsection{First statement} \label{fsec}

Our relations in the tautological ring  $R^*({\mathcal{M}}_g)$
obtained from the moduli of stable quotients are based on the function
\begin{equation}\label{jsjs}
\Phi(t,x) = \sum_{d=0}^\infty \prod_{i=1}^d \frac{1}{1-it} \ 
\frac {(-1)^d}{d!} \frac{x^d}{t^{d}} \ .
\end{equation}
Define the coefficients $\widetilde{C}^d_{r}$ by the logarithm,
$$\log(\Phi)= \sum_{d=1}^\infty \sum_{r=-1}^\infty \widetilde{C}^d_{r}\ t^r \frac{x^d}{d!}
\ .$$
Again, by an application of Wick's formula in Section \ref{pop}, the
$t$ dependence has at most a simple pole.
Let
\begin{equation}\label{f444}
\widetilde{\gamma}=  \sum_{i\geq 1} \frac{B_{2i}}{2i(2i-1)} \kappa_{2i-1} t^{2i-1}
+ 
\sum_{d=1}^\infty \sum_{r=-1}^\infty \widetilde{C}^d_r \kappa_r t^r \frac{x^d}{d!}\ .
\end{equation}
Denote the $t^rx^d$ coefficient of $\exp(-\widetilde{\gamma})$ by
$$\big[ \exp(-\widetilde{\gamma}) \big]_{t^rx^d} \in \mathbb{Q}[\kappa_{-1},
\kappa_0,\kappa_1,
\kappa_2, \ldots] \ .$$
In fact, $[ \exp(-\widetilde{\gamma})]_{t^rx^d}$ is homogeneous of degree $r$
in the $\kappa$ classes.

The first form of the tautological relations obtained from the
moduli of stable quotients
is given by the following result.

\begin{Proposition} \label{vtw}
In $R^r({\mathcal{M}}_g)$, the relation
\begin{equation*}
\big[ \exp(-\widetilde{\gamma}) \big]_{t^rx^d} =0  
\end{equation*}
holds when $g-2d-1< r$ and 
$g\equiv r+1 \hspace{-5pt} \mod 2$.
\end{Proposition}

For fixed $r$ and $d$,
if Proposition \ref{vtw} applies in genus $g$, 
then Proposition \ref{vtw} applies in genera $h=g-2\delta$ for all
natural numbers
$\delta\in \mathbb{N}$.
The genus shifting mod 2 property is 
present also in the Faber-Zagier relations.

\subsection{$K$-theory class $\mathbb{F}_d$}

For genus $g \geq 2$, we consider as before
$$\pi^d: \mathcal{C}^d_g \rarr \MM_g\ , $$
the $d$-fold product of the
universal curve over $M_g$. 
Given an element
$$[C, {p}_1,
\ldots,{p}_d] \in \mathcal{C}^d_g \ , $$
there is a canonically associated stable quotient
\begin{equation}\label{jwq2}
0 \rarr \oh_C(-\sum_{j=1}^d {p}_j) \rarr \oh_C \rarr Q \rarr 0.
\end{equation}
Consider the universal curve
$$\epsilon: U \rarr {\mathcal{C}}^d_{g}$$
with universal quotient sequence
$$0 \rarr S_U \rarr \oh_U \rarr Q_U \rarr 0$$
obtained from \eqref{jwq2}.
Let
$$\mathbb{F}_d= -R\epsilon_*(S^*_U) \in K(\mathcal{C}^d_g)$$
be the class in $K$-theory.
For example,
$$\mathbb{F}_0 = \mathbb{E}^*-\com$$
is the dual of the Hodge bundle minus a rank 1 trivial
bundle. 

By Riemann-Roch,
the rank of 
$\mathbb{F}_d$ is 
$${r}_g(d)=g-d-1.$$ 
However, $\mathbb{F}_d$ is not always represented by a bundle. 
By the derivation of \cite[Section 4.6]{MOP},
\begin{equation}\label{laloo1}
\mathbb{F}_d = \mathbb{E}^*- \mathbb{B}_d - \com,
\end{equation} where 
$\mathbb{B}_d$ has fiber
$H^0(C,\oh_C(\sum_{j=1}^d {p}_j)|_{\sum_{j=1}^d {p}_j})$
over $[C, {p}_1,\ldots,{p}_d].$

The Chern classes of $\mathbb F_d$ can be easily computed. 
Recall the divisor $D_{i,j}$ where the markings $p_i$ 
and $p_j$ coincide. Set $$\Delta_i=D_{1,i}+\ldots+D_{i-1,i},$$ 
with the convention $\Delta_1=0.$ 
Over $[C,  p_1, \ldots,  p_d],$ the virtual 
bundle $\mathbb F_d$ is the formal difference 
$$H^1(\mathcal O_C( p_1+\ldots+ p_d))
-H^0(\mathcal O_C(p_1+\ldots+p_d)).$$ 
Taking the cohomology of the exact sequence 
$$0\to \mathcal O_C( p_1+\ldots+ p_{d-1})\to 
\mathcal O_C(p_1+\ldots+p_d)\to  
\mathcal O_C(p_1+\ldots+p_d)|_{\widehat p_d}\to 0,$$ 
we find $$c(\mathbb F_d)=
\frac{c(\mathbb F_{d-1})}{1+\Delta_d- \psi_d}.$$ 
Inductively, we obtain
 \begin{equation*}c(\mathbb F_d)=
\frac{c(\mathbb E^*)}{(1+\Delta_1-{\psi}_1)\cdots 
(1+\Delta_d-{\psi}_d)}.\end{equation*}
Equivalently, we have
\begin{equation}\label{laloo2}c(-\mathbb B_d)=
\frac{1}{(1+\Delta_1-{\psi}_1)\cdots 
(1+\Delta_d-{\psi}_d)}.\end{equation}

\subsection{Proof of Proposition \ref{vtw}} \label{pop}

Consider the proper morphism
$$\nu: Q_{g}(\proj^1,d) \rarr M_g.$$
%The universal curve 
%$$\pi:U \rarr Q_{g}(\proj^1,d)$$ carries the  basic
%divisor classes
%$${s} = c_1(S_U^*), \ \ \ \  \omega= c_1(\omega_\pi)$$
%obtained from the universal subsheaf $S_U$ and the $\pi$-relative duali%zing
%sheaf. 
Certainly the
class
\begin{equation}\label{p236}
\nu_*\left( 0^c \cap [Q_g(\proj^1,d)]^{vir} \right)
\in A^*(\MM_g,\mathbb{Q}),
\end{equation}
where $0$ is the first Chern class of the trivial bundle,
 vanishes if $c>0$. 
Proposition \ref{vtw} is proven by calculating \eqref{p236} by localization.
We will find Proposition \ref{vtw} is a subset of the much richer
family of relations of Theorem \ref{mmnn} of Section \ref{exrel}.

Let the torus $\com^*$ act on a 2-dimensional vector
space $V\stackrel{\sim}{=}\com^2$ with diagonal weights
$[0,1]$.  The $\com^*$-action lifts
canonically to 
$\proj(V)$
and
$Q_{g}(\proj(V),d)$.
We lift the $\com^*$-action to a rank 1 trivial bundle on $Q_g(\proj(V), d)$
by specifying fiber weight $1$.
The choices determine a $\com^*$-lift of the
class 
$$
 0^c \cap [Q_g(\proj(V),d)]^{vir}\in 
A_{2d+2g-2-c}(
Q_g(\proj(V),d),\mathbb{Q}).$$
 
The push-forward  \eqref{p236} is determined by 
the virtual localization formula \cite{GrP}. There are
only two $\com^*$-fixed loci.
The first corresponds to a vertex lying over $0\in \proj(V)$.
The locus is isomorphic to
$$\mathcal{C}^d_g\ /\ \mathbb{S}_d$$
 and the
associated subsheaf \eqref{jwq2}
lies in the first factor of
$V \otimes  \oh_C$ when considered as a stable quotient in
the moduli space $Q_g(\proj(V),d)$.
Similarly, the second fixed locus 
corresponds to a vertex lying over $\infty\in \proj(V)$.

The localization contribution of the first locus to \eqref{p236} is 
$$\frac{1}{d!}\pi^d_*
\left( c_{g-d-1+c}(\mathbb{F}_d)\right)\  \  \ \ \text{where} \ \ \ \
\pi^d: \mathcal{C}^d_g \rightarrow \MM_g\ . $$
Let $c_-(\mathbb{F}_d)$ denote the total Chern class of $\mathbb{F}_d$
evaluated at $-1$.
%$$c_-(\mathbb{F}_d) = \frac{c_-(\mathbb{E}^*)}{c_-(\mathbb{B}_d)}
%=\frac{1+\lambda_1+\lambda_2+\ldots +\lambda_g}
%{1-\beta_1 +\beta_2 -\ldots+(-1)^d
%\beta_d}\ .$$
The localization contribution of the second locus is
$$\frac{(-1)^{g-d-1}}{d!}\pi^d_*\Big[
 c_{-}(\mathbb{F}_d)\Big]^{g-d-1+c}$$
where $[\gamma]^k$ is the part of $\gamma$ in $A^k(
\mathcal{C}^d_g,\mathbb{Q})$.

Both localization contributions are found by straightforward
expansion of the vertex formulas of \cite[Section 7.4.2]{MOP}.
Summing the contributions yields 
\begin{multline*}
\pi^d_*\Big(
 c_{g-d-1+c}(\mathbb{F}_d) + 
(-1)^{g-d-1} \Big[
 c_{-}(\mathbb{F}_d) \Big] ^{g-d-1+c}
\Big) = 0 \ \ \ 
\text{in  }\ R^*(\MM_g)\ 
\end{multline*}
for $c>0$. We obtain the following result.

\begin{Lemma} \label{htht}
For  $c>0$ and $c\equiv 0 \mod 2$,
\begin{equation*}
\pi^d_*\Big(
 c_{g-d-1+c}(\mathbb{F}_d) 
\Big) = 0 \ \ \ 
\text{in  }\ R^*(\MM_g)\ .
\end{equation*}
\end{Lemma}

For $c>0$, the relation of Lemma \ref{htht} lies in $R^r(\MM_g)$
where
$$r=g-2d-1+c\ .$$
Moreover, the relation is trivial unless
\begin{equation} \label{mss3}
g-d-1 \equiv g-d-1+c = r-d \ \mod 2\ .
\end{equation}

We may expand the right side of \eqref{laloo2} fully. 
The resulting expression is a  polynomial in the 
$d+ \binom{d}{2}$ variables. 
$${\psi}_1,\ldots, {\psi}_d, -D_{12},-D_{13}, 
\ldots,- D_{d-1,d}\ .$$ 
Let $\widetilde{M}_r^d$ denote the coefficient in degree $r$,
$$c_t( -\mathbb{B}_d) =\sum_{r=0}^\infty
\widetilde{M}_r^d({\psi}_i,- D_{ij}) \ t^r.$$
Let $\widetilde{S}^d_r$ be the summand of the 
evaluation  $\widetilde{M}^d_r({\psi}_i=1, -D_{ij}=1)$ 
consisting of the contributions of
only the connected monomials.

\begin{Lemma} \label{llggg}
We have
$$\sum_{d=1}^\infty \sum_{r=0}^\infty
\widetilde{S}_r^d\  t^r \frac{x^d}{d!}  =
\log\left( 1+\sum_{d=1}^\infty \prod_{i=1}^d \frac{1}{1-it}
\ \frac{x^d}{d!}\right)\ .
$$
\end{Lemma}

\begin{proof}
As before, by Wick's formula, the connected
and disconnected counts are related by exponentiation,
$$\exp\left(\sum_{d=1}^\infty \sum_{r=0}^\infty
\widetilde{S}_r^d \ t^r\frac{x^d}{d!}\right) =
1+ \sum_{d=1}^\infty \sum_{r=0}^\infty
\widetilde{M}_r^d(\widehat{\psi}_i=1, -D_{ij}=1) \ t^r\frac{x^d}{d!} \ .$$
\end{proof}

Since a connected monomial in the variables $\psi_i$ and $-D_{ij}$
must have at least $d-1$ factors of the variables $-D_{ij}$,
we see $\widetilde{S}^d_r =0$ if $r<d-1$. Using the self-intersection
formulas, we obtain
\begin{equation}
\sum_{d=1}^ \infty \sum_{r\geq 0} \pi^d_*\big(c_{r}(-\mathbb{B}_d)\big)\ t^r\frac{x^d}{d!}
= \exp\left(\sum_{d=1}^\infty \sum_{r=0}^\infty
\widetilde{S}_r^d (-1)^{d-1}\kappa_{r-d}\ t^r\frac{x^d}{d!}\right) \ . 
\end{equation}
To account for the alternating factor $(-1)^{d-1}$ and the $\kappa$ subscript,
we define the coefficients $\widetilde{C}^d_r$ by 
$$\sum_{d=1}^\infty \sum_{r\geq -1}
\widetilde{C}_r^d\  t^r \frac{x^d}{d!}  =
\log\left( 1+\sum_{d=1}^\infty \prod_{i=1}^d \frac{1}{1-it}\ 
\frac{(-1)^d}{t^d} \frac{x^d}{d!}
\right)\ .
$$
The vanishing  $\widetilde{S}^d_{r<d-1}=0$ implies the vanishing $\widetilde{C}^d_{r<-1}=0$. 

Again using
Mumford's Grothendieck-Riemann-Roch calculation \cite{M},
$$c_t(\mathbb{E}^*) = -\sum_{i\geq 1} \frac{B_{2i}}{2i(2i-1)} \kappa_{2i-1} t^{2i-1}\ .$$
Putting the above results together yields the following 
formula:
\begin{multline*}
\sum_{d=1}^ \infty \sum_{r\geq 0} \pi^d_*\big(c_r(\mathbb{F}_d)\big)\ t^{r-d}\frac{x^d}{d!}
= \\
\exp\left(-\sum_{i\geq 1} \frac{B_{2i}}{2i(2i-1)} \kappa_{2i-1} t^{2i-1}
-\sum_{d=1}^\infty \sum_{r\geq -1}
\widetilde{C}_r^d \kappa_r t^r \frac{x^d}{d!}
\right) \ .
\end{multline*}
The restrictions on $g$, $d$, and $r$ in the statement of
Proposition \ref{vtw} are obtained from \eqref{mss3}. \qed

\subsection{Extended relations} \label{exrel}
The universal curve 
$$\epsilon:U \rarr Q_{g}(\proj^1,d)$$ carries the  basic
divisor classes
$${s} = c_1(S_U^*), \ \ \ \  \omega= c_1(\omega_\pi)$$
obtained from the universal subsheaf $S_U$ 
of the moduli of stable quotients 
and the $\epsilon$-relative 
dualizing
sheaf. 
Following \cite[Proposition 5]{MOP}, we
can obtain a much larger set of relations in the tautological
ring of $\MM_g$ by including factors of $\epsilon_*(s^{a_i}\omega^{b_i})$
in the integrand:
\begin{equation*}
\nu_*\left(\prod_{i=1}^n\epsilon_*(s^{a_i} \omega^{b_i}) \cdot  0^c \cap [Q_g(\proj^1,d)]^{vir} \right)= 0\ \ 
\text{in} \  A^*(\MM_g,\mathbb{Q}) \ 
\end{equation*}
when $c>0$.
We will study the
associated relations where the $a_i$ are always $1$.
The $b_i$ then form the parts of a partition $\sigma$.

To state the relations we obtain, 
we start by extending the function $\widetilde{\gamma}$ of Section \ref{fsec},
\begin{eqnarray*}
\gamma^{\text{\tiny{{\sf SQ}}}} &=& 
 \sum_{i\geq 1} \frac{B_{2i}}{2i(2i-1)} \kappa_{2i-1} t^{2i-1}\\
& &
+ 
\sum_{\sigma}
\sum_{d=1}^\infty \sum_{r=-1}^\infty \widetilde{C}^{d}_r 
\kappa_{r+|\sigma|}\ t^r \frac{x^d}{d!} \ \frac{d^{\ell(\sigma)} t^{|\sigma|}
{\mathbf{p}}^{\sigma}}
{|\text{Aut}(\sigma)|}
\ .
\end{eqnarray*}
Let $\overline{\gamma}^{\, \text{\tiny{{\sf SQ}}}}$ be defined by a similar formula,
\begin{eqnarray*}
\overline{\gamma}^{\, \text{\tiny{{\sf SQ}}}} &=& 
 \sum_{i\geq 1} \frac{B_{2i}}{2i(2i-1)} \kappa_{2i-1} (-t)^{2i-1}\\
& &
+ 
\sum_{\sigma}
\sum_{d=1}^\infty \sum_{r=-1}^\infty \widetilde{C}^{d}_r 
\kappa_{r+|\sigma|}\ (-t)^r \frac{x^d}{d!} 
\ \frac{d^{\ell(\sigma)} t^{|\sigma|}
{\mathbf{p}}^{\sigma}}
{|\text{Aut}(\sigma)|}
\ .
\end{eqnarray*}
The sign of $t$ in $t^{|\sigma|}$ does not change in 
$\overline{\gamma}^{\, \text{\tiny{{\sf SQ}}}}$.
The $\kappa_{-1}$ terms which appear will later be set to
0.

The full system of relations are obtained from the 
coefficients of the functions
\vspace{-10pt}
$$\exp(-\gamma^{\text{\tiny{{\sf SQ}}}}), \ \ \ \ 
\exp(-\sum_{r=0}^\infty \kappa_r t^r p_{r+1})\cdot
\exp( -\overline{\gamma}^{\, \text{\tiny{{\sf SQ}}}})
$$

\begin{Theorem} \label{mmnn} In $R^r(\MM_g)$, the relation
$$\Big[ \exp(-\gamma^{\text{\tiny{{\sf SQ}}}}) \Big]_{t^rx^d\mathbf{p}^\sigma} =
(-1)^g\Big[ 
\exp(-\sum_{r=0}^\infty \kappa_r t^r p_{r+1})\cdot
\exp( -\overline{\gamma}^{\, \text{\tiny{{\sf SQ}}}})
 \Big]_{t^rx^d\mathbf{p}^\sigma}$$
holds when $g-2d-1+|\sigma| < r$.
\end{Theorem}

Again, we see the genus shifting mod 2 property.
If the relation  holds in genus $g$, then the
{\em same} relation holds in genera $h=g-2\delta$ for all
natural numbers
$\delta\in \mathbb{N}$.

In case $\sigma=\emptyset$, Theorem \ref{mmnn} specializes to 
the relation
%{\footnote{We omit the superscript \sf{SQ}
%for notation convenience when no confusion is  expected.}}
\begin{eqnarray*}
\Big[ \exp(-\widetilde{\gamma}(t,x)) \Big]_{t^rx^d} & = &
(-1)^g\Big[ 
\exp( -\widetilde{\gamma}(-t,x))
 \Big]_{t^rx^d} \\
& = & (-1)^{g+r} \Big[ 
\exp( -\widetilde{\gamma}(t,x))
 \Big]_{t^rx^d}\ ,
\end{eqnarray*}
nontrivial only if $g\equiv r+1$ mod 2. If the mod 2 condition
holds, then we obtain the relations of
Proposition \ref{vtw}.

Consider the case $\sigma=(1)$. The left side of the relation is
then
$$
\Big[ \exp(-\widetilde{\gamma}(t,x))\cdot \left(-\sum_{d=1}^\infty
\sum_{s=-1}^\infty \widetilde{C}^d_{s}\ \kappa_{s+1} t^{s+1} \frac{dx^d}{d!}\right)
\Big]_{t^rx^d} \ .
$$
The right side is
$$(-1)^g\Big[ \exp(-\widetilde{\gamma}(-t,x))\cdot\left(-\kappa_0t^0+
\sum_{d=1}^\infty
\sum_{s=-1}^\infty \widetilde{C}^d_{s}\ \kappa_{s+1} (-t)^{s+1} \frac{dx^d}{d!}
\right)
\Big]_{t^rx^d} \ .
$$

If $g\equiv r+1$ mod 2, then the large terms cancel and we
obtain
$$-\kappa_0 \cdot \Big[ \exp(-\widetilde{\gamma}(t,x)) \Big]_{t^rx^d} =0 \ . $$
Since $\kappa_0=2g-2$ and 
$$(g-2d-1+1 < r)   \ \ \implies\ \  (g-2d-1 < r),$$
we recover most (but not all)
of the $\sigma=\emptyset$ equations.

If $g\equiv r$ mod 2, then the resulting equation is
\begin{equation*}
\Big[ \exp(-\widetilde{\gamma}(t,x))\cdot \left(\kappa_0-2
\sum_{d=1}^\infty
\sum_{s=-1}^\infty \widetilde{C}^d_{s}\ \kappa_{s+1} t^{s+1} \frac{dx^d}{d!}\right)
\Big]_{t^rx^d}=0
\end{equation*}
when $g-2d<r$.

\subsection{Proof of Theorem \ref{mmnn}}

\subsubsection{Partitions, differential operators, and logs.}
\label{pdol}

We will write partitions $\sigma$ as $(1^{n_1}2^{n_2}3^{n_3}\ldots)$
with
$$\ell(\sigma)= \sum_{i} n_i \ \ \ \ \text{and} \ \ \ \
|\sigma|= \sum_i in_i\ .$$
The empty partition $\emptyset$ corresponding to
$(1^{0}2^{0}3^{0}\ldots)$
is permitted. In all cases, we have
$$|{\text{Aut}}(\sigma)|= n_1!n_2!n_3! \cdots \ .$$
In the 
 infinite set of  variables
$\{ p_1, p_2, p_3, \ldots \}$,
let
$$\Phi^{\mathbf{p}}(t,x) = 
\sum_\sigma \sum_{d=0}^\infty  
\prod_{i=1}^d \frac{1}{1-it} \ 
\frac {(-1)^d}{d!}  \frac{x^d}{t^{d}}\  
\frac{d^{\ell(\sigma)} t^{|\sigma|}
{\mathbf{p}}^{\sigma}}
{|\text{Aut}(\sigma)|}
\ ,$$
where the first sum is over all partitions $\sigma$.
The summand corresponding to the empty partition
equals $\Phi(t,x)$ defined in \eqref{jsjs}.

The function $\Phi^{\mathbf{p}}$ is easily obtained
from $\Phi$,
$$\Phi^{\mathbf{p}}(t,x) = \exp\left( \sum_{i=1}^\infty p_it^i x\frac{d}{dx}
\right)
\ \Phi(t,x)\ . $$
Let $D$ denote the differential operator
$$D= \sum_{i=1}^\infty p_it^i x\frac{d}{dx}\ .$$
Expanding the exponential of $D$, we obtain
\begin{eqnarray}
\Phi^{\mathbf{p}}& = & \Phi + D\Phi + \frac{1}{2} D^2\Phi+
\frac{1}{6} D^3 \Phi+\ldots \label{vfrr}
\\
& =& \nonumber
\Phi \left(1+\frac{D \Phi}{\Phi} +
\frac{1}{2} \frac{D^2 \Phi}{\Phi} + 
\frac{1}{6}\frac{D^3\Phi}{\Phi}+
\ldots\right) \ .
\end{eqnarray}
Let $\gamma^* = \log (\Phi)$ be the logarithm, 
 $$D\gamma^* = \frac{D\Phi}{\Phi}\ .$$
After applying the logarithm to \eqref{vfrr}, we see
\begin{eqnarray*}
\log(\Phi^{\mathbf{p}}) & =&
\gamma^* +\log\left(
 1+ D \gamma^* +  \frac{1}{2}( D^2\gamma^* + (D\gamma^*)^2)+ \ ... \right) \\
& = & \gamma^* + D\gamma^* + \frac{1}{2} D^2 \gamma^* + \ldots
\end{eqnarray*}
where the dots stand for a universal expression in
the $D^k\gamma^*$.
In fact, a remarkable simplification occurs,
$$\log(\Phi^{\mathbf{p}}) = 
\exp\left( \sum_{i=1}^\infty p_it^i x\frac{d}{dx}
\right) \ \gamma^*\ .$$
The result follows from a general identity.

\begin{Proposition} \label{vwwv3}
If $f$ is   a function of $x$, then
$$\log\left(\exp\left(\lambda x\frac{d}{dx}\right) \ f \right) =
\exp\left(\lambda x\frac{d}{dx}\right) \ \log(f)\ .$$
\end{Proposition}

\begin{proof}
A simple computation for monomials in $x$ shows
$$\exp\left(\lambda x\frac{d}{dx}\right) \ x^k = (e^\lambda x)^k\  .$$
Hence, since the differential operator is additive,
$$\exp\left(\lambda x\frac{d}{dx}\right) \ f(x) = f(e^\lambda x)\ .$$
The Proposition follows immediately.
\end{proof}

We apply Proposition \ref{vwwv3} to $\log(\Phi^{\mathbf{p}})$.
The coefficients of the logarithm may be written as 
\begin{eqnarray*}
\log(\Phi^{\mathbf{p}}) & = &
\sum_\sigma \sum_{d=1}^\infty \sum_{r=-1}^\infty \widetilde{C}^d_{r}(\sigma) \ t^r \frac{x^d}{d!} {{\mathbf{p}}^{\sigma}}
\\
& = &
\sum_{d=1}^\infty \sum_{r=-1}^\infty \widetilde{C}^d_{r}\ t^r \frac{x^d}{d!}
\exp\left( \sum_{i=1}^\infty dp_i t^i\right)\\
& = & 
\sum_{\sigma} \sum_{d=1}^\infty \sum_{r=-1}^\infty 
\widetilde{C}^d_{r}\ t^r \frac{x^d}{d!}
\ \frac{d^{\ell(\sigma)} t^{|\sigma|}
{\mathbf{p}}^{\sigma}}
{|\text{Aut}(\sigma)|}
\ . 
\end{eqnarray*}
We have expressed the coefficients $\widetilde{C}^d_{r}(\sigma)$
of $\log(\Phi^{\mathbf{p}})$ solely in terms of the
coefficients
$\widetilde{C}^d_{r}$ of $\log(\Phi)$.

\subsubsection{Cutting classes}
Let $\theta_i\in A^1(U,\mathbb{Q})$ be the
class of the $i^{th}$ section of the universal curve 
\begin{equation}\label{gbbt}
\epsilon: U \rarr \mathcal{C}_{g}^d 
\end{equation}
 The class $s=c_1(S_U^*)$ on the universal curve over $Q_g(\mathbf{P}^1,d)$ 
restricted to the $\com^*$-fixed locus $\mathcal{C}^d_g / \mathbb{S}_d$
and pulled-back to 
\eqref{gbbt} yields 
$$s=\theta_1 + \ldots +\theta_d\ \in A^1(U,\mathbb{Q}).$$
We calculate 
\begin{equation}\label{cutt3}
\epsilon_*(s\ \omega^b)  =  {\psi}^b_1 +\ldots +{\psi}^b_d \ \ \in
A^b(\mathcal{C}_g^d,\mathbb{Q})\ .
\end{equation}

\subsubsection{Wick form}
We repeat the Wick analysis of Section \ref{pop} for the vanishings
\begin{equation*}
\nu_*\left(\prod_{i=1}^\ell\epsilon_*(s \omega^{b_i}) \cdot  0^c \cap [Q_g(\proj^1,d)]^{vir} \right) = 0 \ \
{\text {in}}\  A^*(\MM_g,\mathbb{Q}) 
\end{equation*} 
when $c>0$.
We start by writing a formula
for
\begin{equation*}
\sum_{d=1}^ \infty \sum_{r\geq 0} \pi^d_*\left(
\exp\Big(\sum_{i=1}^\infty p_it^i \epsilon_*(s\omega^i)  \Big) \cdot
c_r(\mathbb{F}_d) t^r\right)\ \frac{1}{t^d}\frac{x^d}{d!} \ .
\end{equation*}

Applying
the Wick formula 
to equation \eqref{cutt3} for the cutting classes, 
we see
\begin{equation}\label{mssx2}
\sum_{d=1}^ \infty \sum_{r\geq 0} \pi^d_*\left(
\exp\Big(\sum_{i=1}^\infty p_it^i \epsilon_*(s\omega^i)  \Big) \cdot
c_r(\mathbb{F}_d) t^r\right)\ \frac{1}{t^d}\frac{x^d}{d!}  = 
\exp(-\widetilde{\gamma}^{\, \text{\tiny{{\sf SQ}}}})
\end{equation}
where $\widetilde{\gamma}^{\, \text{\tiny{{\sf SQ}}}}$ is defined by
\begin{equation*}
\widetilde{\gamma}^{\, \text{\tiny{{\sf SQ}}}} = 
 \sum_{i\geq 1} \frac{B_{2i}}{2i(2i-1)} \kappa_{2i-1} t^{2i-1} 
+ 
\sum_{\sigma}
\sum_{d=1}^\infty \sum_{r=-1}^\infty \widetilde{C}^{d}_r(\sigma) 
\kappa_{r}\ t^r \frac{x^d}{d!} \ 
{{\mathbf{p}}^{\sigma}}
\ .
\end{equation*}
We follow here the notation of Section \ref{pdol},
$$\Phi^{\mathbf{p}}(t,x) = 
\sum_\sigma \sum_{d=0}^\infty  
\prod_{i=1}^d \frac{1}{1-it} \ 
\frac {(-1)^d}{d!}  \frac{x^d}{t^{d}}\  
\frac{d^{\ell(\sigma)} t^{|\sigma|}
{\mathbf{p}}^{\sigma}}
{|\text{Aut}(\sigma)|}
\ ,$$
$$ 
\log(\Phi^{\mathbf{p}}) = 
\sum_\sigma \sum_{d=1}^\infty \sum_{r=-1}^\infty \widetilde{C}^d_{r}(\sigma) \ t^r \frac{x^d}{d!} {{\mathbf{p}}^{\sigma}}
\ .$$
In the Wick analysis, the class $\epsilon_*(s \omega^b)$
simply acts as $dt^b$.

Using the expression for the coefficents $\widetilde{C}^d_r(\sigma)$
in terms of $\widetilde{C}^d_r$ derived in Section \ref{pdol},
we obtain the following result from \eqref{mssx2}.

\begin{Proposition} \label{mmnn3}
We have 
\begin{equation*}
\sum_{d=1}^ \infty \sum_{r\geq 0} \pi^d_*\left(
\exp\Big(\sum_{i=1}^\infty p_it^i \epsilon_*(s\omega^i)  \Big) \cdot
c_r(\mathbb{F}_d) t^r\right)\ \frac{1}{t^d}\frac{x^d}{d!}  = 
\exp(-{\gamma}^{\text{\tiny{{\sf SQ}}}}) \ .
\end{equation*}
\end{Proposition}

\subsubsection{Geometric construction}
We apply $\com^*$-localization on $Q_g(\mathbf{P}^1,d)$
to the geometric vanishing
\begin{equation} \label{frad}
\nu_*\left(\prod_{i=1}^\ell\epsilon_*(s \omega^{b_i}) \cdot  0^c \cap [Q_g(\proj^1,d)]^{vir} \right)= 0\ \ 
\text{in} \  A^*(\MM_g,\mathbb{Q}) \ 
\end{equation}
when $c>0$.
The result is the relation
\begin{multline} \label{gred}
\pi_*\Big(
\prod_{i=1}^\ell \epsilon_*(s \omega^{b_i})\cdot c_{g-d-1+c}(\mathbb{F}_d) + \\
(-1)^{g-d-1} \Big[ \prod_{i=1}^\ell
\epsilon_*\left((s-1)\omega^{b_i}\right)
\cdot c_{-}(\mathbb{F}_d) \Big] ^{g-d-1+\sum_ib_i+c}
\Big) = 0
\end{multline}
in $R^*(\MM_g)$.
After applying the Wick formula of Proposition \ref{mmnn3}, we immediately
obtain Theorem \ref{mmnn}.

The first summand in \eqref{gred} yields the left side
$$\Big[ \exp(-\gamma^{\text{\tiny{{\sf SQ}}}}) \Big]_{t^rx^d\mathbf{p}^\sigma} $$
of the relation of Theorem \ref{mmnn}.
The second summand produces the right side
\begin{equation}\label{kedd}
(-1)^g\Big[ 
\exp(-\sum_{r=0}^\infty \kappa_r t^r p_{r+1})\cdot
\exp( -\widehat{\gamma}^{\, \text{\tiny{{\sf SQ}}}})
 \Big]_{t^rx^d\mathbf{p}^\sigma}\ .
\end{equation}
Recall the localization of the virtual class over $\infty \in \mathbf{P}^1$
is 
$$\frac{(-1)^{g-d-1}}{d!}\pi^d_*\Big[
 c_{-}(\mathbb{F}_d)\Big]^{g-d-1+c}\ .$$
Of the sign prefactor $(-1)^{g-d-1}$,
\begin{enumerate}
\item[$\bullet$] $(-1)^{-1}$ is used to move the
term to the right side, 
\item[$\bullet$]
$(-1)^{-d}$ is absorbed in the $(-t)$
of the definition of $\widehat{\gamma}^{\, \text{\tiny{{\sf SQ}}}}$,
\item[$\bullet$]
$(-1)^g$ remains in \eqref{kedd}.
\end{enumerate}
The $-1$ of $s-1$ produces the 
the factor $\exp(-\sum_{r=0}^\infty \kappa_r t^r p_{r+1})$.

Finally, a simple dimension calculation (remembering
$c>0$) implies the validity of the relation
 when $g-2d-1+|\sigma| < r$. \qed

\section{Analysis of the relations}
\label{LLL}

\subsection{Expanded form} \label{exf}

Let $\sigma=(1^{a_1}2^{a_2}3^{a_3} \ldots)$ be a partition
of length $\ell(\sigma)$ and size $|\sigma|$.
We can directly write the corresponding
tautological relation in $R^r(\MM_g)$
obtained from Theorem \ref{mmnn}.

A {\em subpartition} $\sigma'\subset \sigma$ is obtained
by selecting a nontrivial subset of the parts of $\sigma$.
A {\em division} of $\sigma$ is a disjoint union
\begin{equation}\label{rrgg}
\sigma = \sigma^{(1)} \cup \sigma^{(2)} \cup \sigma^{(3)}\ldots
\end{equation}
of subpartitions which exhausts $\sigma$.
The subpartitions in \eqref{rrgg} are unordered.
Let
$\mathcal{S}(\sigma)$ be the set of divisions of $\sigma$.
For example,
\begin{eqnarray*}
\mathcal{S}(1^12^1) &=& \{ \ (1^12^1),\ (1^1) \cup (2^1)\ \}\ , \\
\mathcal{S}(1^3) &=& \{\ (1^3), \ (1^2)\cup (1^1) \ \}\ .
\end{eqnarray*}

We will use the notation $\sigma^\bullet$ to
denote a division of $\sigma$ with subpartitions  $\sigma^{(i)}$.
Let
$$m(\sigma^\bullet) = \frac{1}{|\text{Aut}(\sigma^\bullet)|}
\frac{|\text{Aut}(\sigma)|}{\prod_{i=1}^{\ell(\sigma^\bullet)}
|\text{Aut}(\sigma^{(i)})|}.$$
Here, $\text{Aut}(\sigma^\bullet)$ is the group
permuting equal subpartitions.
The factor $m(\sigma^\bullet)$ may be interpreted
as counting the number of different ways the disjoint union
can be made.

To write explicitly the $\mathbf{p}^\sigma$ coefficient of
$\exp(\gamma^{\text{\tiny{{\sf SQ}}}})$, we introduce the functions 
$$F_{n,m}(t,x) = -\sum_{d=1}^\infty \sum_{s=-1}^\infty \widetilde{C}^d_{s}\ \kappa_{s+m} t^{s+m} 
\frac{d^n x^d}{d!}$$
for $n,m \geq 1$.
Then,
\begin{multline*}
|\text{Aut}(\sigma)| \cdot
\Big[ \exp(-\gamma^{\text{\tiny{{\sf SQ}}}}) \Big]_{t^rx^d\mathbf{p}^\sigma} =\\
\Big[ \exp(-\widetilde{\gamma}(t,x)) 
\cdot
\left(
\sum_{\sigma^\bullet\in \mathcal{S}(\sigma)}
m(\sigma^\bullet)
\prod_{i=1}^{\ell(\sigma^\bullet)}
F_{\ell(\sigma^{(i)}), |\sigma^{(i)}|} \right) 
\Big]_{t^rx^d} \ .
\end{multline*}

Let $\sigma^{*,\bullet}$ be a division of $\sigma$ with a 
marked subpartition,
\begin{equation}\label{rrggg}
\sigma = \sigma^* \cup \sigma^{(1)} \cup \sigma^{(2)} \cup \sigma^{(3)}\ldots,
\end{equation}
labelled by the superscript $*$. 
The marked subpartition is permitted to be empty. Let 
$\mathcal{S}^*(\sigma)$ denote the
set of marked divisions of $\sigma$.
Let
$$m(\sigma^{*,\bullet}) = \frac{1}{|\text{Aut}(\sigma^\bullet)|}
\frac{|\text{Aut}(\sigma)|}{|\text{Aut}(\sigma^*)|
\prod_{i=1}^{\ell(\sigma^{*,\bullet})}
|\text{Aut}(\sigma^{(i)})|}.$$
The length $\ell(\sigma^{*,\bullet})$ is the number
of unmarked subpartitions.

Then, $|\text{Aut}(\sigma)|$ 
times the right side of Theorem \ref{mmnn} may be written as
\begin{multline*}
(-1)^{g+|\sigma|} |\text{Aut}(\sigma)| \cdot
\Big[ \exp(-\widetilde{\gamma}(-t,x)) 
\cdot\\
\left(
\sum_{\sigma^{*,\bullet}\in \mathcal{S}^*(\sigma)}
m(\sigma^{*,\bullet})
 \prod_{j=1}^{\ell(\sigma^*)}
\kappa_{\sigma^*_{j}-1} (-t)^{\sigma^*_j-1}
\prod_{i=1}^{\ell(\sigma^{*,\bullet})}
F_{\ell(\sigma^{(i)}), |\sigma^{(i)}|} (-t,x)\right)
\Big]_{t^rx^d}
\end{multline*}

To write Theorem \ref{mmnn} in the simplest form, the following
definition using the Kronecker $\delta$ is useful,
$$m^\pm(\sigma^{*,\bullet}) = 
(1\pm\delta_{0,|\sigma^*|}) \cdot m(\sigma^{*,\bullet}).$$
There are two cases.
If $g\equiv r +|\sigma|$ mod 2, then Theorem 3 is equivalent to
the vanishing of
\begin{equation*} 
{\small{|\text{Aut}(\sigma)|}}\Big[ \exp(-\widetilde{\gamma}) 
\cdot
\left(
\sum_{\sigma^{*,\bullet}\in \mathcal{S}^*(\sigma)}
m^-(\sigma^{*,\bullet})
 \prod_{j=1}^{\ell(\sigma^*)}
\kappa_{\sigma^*_{j}-1} t^{\sigma^*_j-1}
\prod_{i=1}^{\ell(\sigma^{*,\bullet})}
F_{\ell(\sigma^{(i)}), |\sigma^{(i)}|} \right)
\Big]_{t^rx^d}  .
\end{equation*}
If $g\equiv r +|\sigma|+1$ mod 2, then Theorem \ref{mmnn} is equivalent
to the vanishing of
\begin{equation*}
{\small{|\text{Aut}(\sigma)|}}
\Big[ \exp(-\widetilde{\gamma}) 
\cdot
\left(
\sum_{\sigma^{*,\bullet}\in \mathcal{S}^*(\sigma)}
m^+(\sigma^{*,\bullet})
 \prod_{j=1}^{\ell(\sigma^*)}
\kappa_{\sigma^*_{j}-1} t^{\sigma^*_j-1}
\prod_{i=1}^{\ell(\sigma^{*,\bullet})}
F_{\ell(\sigma^{(i)}), |\sigma^{(i)}|} \right)
\Big]_{t^rx^d} .
\end{equation*}
In either case, the relations are valid
in the ring $R^*(\MM_g)$ only if the condition
$g-2d-1+|\sigma| < r$ holds.

We denote the above relation corresponding to $g$, $r$, $d$, and
$\sigma$ (and depending upon the
parity of $g-r-|\sigma|$) by
$$\mathsf{R}(g,r,d,\sigma)= 0$$ 
The  $|\text{Aut}(\sigma)|$ prefactor is included
in $\mathsf{R}(g,r,d,\sigma)$, but
is only relevant when $\sigma$ has repeated parts.
In case of repeated parts, the automorphism scaled normalization
is more convenient.

\subsection{Further examples}

If $\sigma=(k)$ has a single part, then the two cases
of Theorem \ref{mmnn} are the following.
If $g\equiv r+k$ mod 2,  we have
$$\Big[ \exp(-\widetilde{\gamma})\cdot \kappa_{k-1}t^{k-1} \Big]_{t^r x^d}=0\ $$
which is a consequence of the $\sigma=\emptyset$ case.
If $g\equiv r+k+1$ mod 2, we have 
$$\Big[ \exp(-\widetilde{\gamma})\cdot \left(\kappa_{k-1}t^{k-1}
+ 2 F_{1,k}\right) \Big]_{t^r x^d}=0$$

If $\sigma=(k_1k_2)$ has two distinct parts, then the two cases
of Theorem \ref{mmnn} are as follows.
If $g\equiv r+k_1+k_2$ mod 2,  we have
\begin{multline*}\Big[ \exp(-\widetilde{\gamma})\cdot \big(
\kappa_{k_1-1}\kappa_{k_2-1}t^{k_1+k_2-2}\\ +\kappa_{k_1-1}t^{k_1-1} F_{1,k_2}
+\kappa_{k_2-1} t^{k_2-1} F_{1,k_1}\big) 
\Big]_{t^r x^d}=0\ .
\end{multline*}
If $g\equiv r+k_1+k_2+1$ mod 2,  we have
\begin{multline*}\Big[ \exp(-\widetilde{\gamma})\cdot \big(
\kappa_{k_1-1}\kappa_{k_2-1}t^{k_1+k_2-2} +\kappa_{k_1-1} t^{k_1-1}F_{1,k_2}
\\+ \kappa_{k_2-1} t^{k_2-1}F_{1,k_1} 
+ 2 F_{2,k_1+k_2} + 2 F_{1,k_1} F_{1,k_2}
\big) 
\Big]_{t^r x^d}=0\ .
\end{multline*}

In fact, the $g\equiv r+k_1+k_2$ mod 2 equation above is not new.
The genus $g$ and codimension $r_1=r-k_2+1$ case of partition $(k_1)$
yields
$$\Big[ \exp(-\widetilde{\gamma})\cdot \left(\kappa_{k_1-1}t^{k_1-1}
+ 2 F_{1,k_1}\right) \Big]_{t^{r_1} x^d}=0\ .$$
After multiplication with $\kappa_{k_2-1}t^{k_2-1}$, we obtain
$$\Big[ \exp(-\widetilde{\gamma})
\cdot \left(\kappa_{k_1-1}\kappa_{k_2-1}t^{k_1+k_2-2}
+ 2 \kappa_{k_2-1}t^{k_2-1} F_{1,k_1}\right) \Big]_{t^{r} x^d}=0\ .$$
Summed with the corresponding equation with $k_1$ and $k_2$
interchanged yields the above 
$g\equiv r+k_1+k_2$ mod 2 case.

\label{vxvx}

\subsection{Expanded form revisited}

Consider the partition $\sigma=(k_1k_2\cdots k_\ell)$
with distinct parts. Relation $\mathsf{R}(g,r,d,\sigma)$, in
the $g\equiv r+|\sigma|$ mod 2 case, is  the
vanishing of 
\begin{multline*}
\Big[ \exp(-\widetilde{\gamma}) 
\cdot
\left(
\sum_{\sigma^{*,\bullet}\in \mathcal{S}^*(\sigma)}
(1-\delta_{0,|\sigma^*|})
 \prod_{j=1}^{\ell(\sigma^*)}
\kappa_{\sigma^*_{j}-1} t^{\sigma^*_j-1}
\prod_{i=1}^{\ell(\sigma^{*,\bullet})}
F_{\ell(\sigma^{(i)}), |\sigma^{(i)}|} \right)
\Big]_{t^rx^d}  
\end{multline*}
since all the factors $m(\sigma^{*,\bullet})$ are
1. In the 
$g\equiv r+|\sigma|+1$ mod 2 case, 
$\mathsf{R}(g,r,d,\sigma)$ is 
 the vanishing
of
\begin{multline*}
\Big[ \exp(-\widetilde{\gamma}) 
\cdot
\left(
\sum_{\sigma^{*,\bullet}\in \mathcal{S}^*(\sigma)}
(1+\delta_{0,|\sigma^*|})
 \prod_{j=1}^{\ell(\sigma^*)}
\kappa_{\sigma^*_{j}-1} t^{\sigma^*_j-1}
\prod_{i=1}^{\ell(\sigma^{*,\bullet})}
F_{\ell(\sigma^{(i)}), |\sigma^{(i)}|} \right)
\Big]_{t^rx^d}  
\end{multline*}
for the same reason.

If $\sigma$ has repeated parts, the relation $\mathsf{R}(g,r,d,\sigma)$
is obtained by viewing the parts as
distinct and specializing the indicies afterwards.
For example,
 the two cases
for $\sigma=(k^2)$
are as follows.
If $g\equiv r+2k$ mod 2,  we have
\begin{equation*}\Big[ \exp(-\widetilde{\gamma})\cdot \big(
\kappa_{k-1}\kappa_{k-1}t^{2k-2} +2\kappa_{k-1}t^{k-1} F_{1,k}\big) 
\Big]_{t^r x^d}=0\ .
\end{equation*}
If $g\equiv r+2k+1$ mod 2,  we have
\begin{multline*}\Big[ \exp(-\widetilde{\gamma})\cdot \big(
\kappa_{k-1}\kappa_{k-1}t^{2k-2} +2\kappa_{k-1} t^{k-1}F_{1,k}
\\ 
+ 2 F_{2,2k} + 2 F_{1,k} F_{1,k}
\big) 
\Big]_{t^r x^d}=0\ .
\end{multline*}
The factors occur via repetition of terms in
the formulas for distinct parts.

\begin{Proposition}
The relation $\mathsf{R}(g,r,d,\sigma)$
in the $g\equiv r+|\sigma|$ mod 2 case is a  consequence of
the relations in  $\mathsf{R}(g,r',d,\sigma')$ where
$g\equiv r'+|\sigma'|+1$ mod 2  and
$\sigma'\subset \sigma$ is a strictly smaller partition.
\end{Proposition}

\begin{proof}
The strategy follows the example of the phenonenon already discussed
in Section \ref{vxvx}. 

%Let $\mathsf{R}(g,r,d,\sigma)$ denote $|\text{Aut}(\sigma)|$ times
%the corresponding 
%relation of Theorem  \ref{mmnn}
%in the expanded form of Section \ref{exf} depending upon the
%parity of $g-r-|\sigma|$.
If $g\equiv r+|\sigma|$ mod 2, then
for every subpartition $\tau\subset \sigma$ of odd length, we have
$$g \equiv r- |\tau|+ \ell(\tau) + |\sigma/\tau| +1 \ \mod \ 2\ $$
where $\sigma/\tau$ is the complement of $\tau$.
The relation
$$\prod_i \kappa_{\tau_i-1} \cdot \mathsf{R}(g,r- |\tau|+ \ell(\tau), d,
\sigma/\tau)$$
is of codimension $r$.

Let $g\equiv r+|\sigma|$ mod 2, and let $\sigma$ have distinct parts. 
The formula
\begin{multline}\label{ggee}
\mathsf{R}(g,r,d,\sigma)= \\ \sum_{\tau\subset \sigma}
\left(\frac{2^{\ell(\tau)+2}-2}{\ell(\tau)+1}\right) B_{\ell(\tau)+1} \cdot 
\prod_i \kappa_{\tau_i-1} \cdot \mathsf{R}\Big(g,r- |\tau|+ \ell(\tau), d,
\sigma/\tau\Big)
\end{multline}
follows easily by grouping like terms and 
the Bernoulli identity
%\begin{eqnarray*}
%n\equiv 0\ \text{mod} \ 2: \ \ \ \sum_{k\geq 1} \binom{n}{2k-1} \left(\frac{2^{%2k+1} -2}{2k}\right) B_{2k}  & = & 0\ ,  \\
%n \equiv 1\ \text{mod} \ 2: \ \ \
%\sum_{k\geq 1} \binom{n}{2k-1} \left(\frac{2^{2k+1} -2}{2k}\right) B_{2k}  & = %& -
% \left(\frac{2^{n+2} -2}{n+1}\right) B_{n+1}\ .
%\end{eqnarray*}
\begin{equation} \label{bpp}
\sum_{k\geq 1} \binom{n}{2k-1} \left(\frac{2^{2k+1} -2}{2k}\right) B_{2k}  = -
 \left(\frac{2^{n+2} -2}{n+1}\right) B_{n+1}\ 
\end{equation}
for $n>0$.
The sum in \eqref{ggee} is over all subpartitions $\tau\subset \sigma$
of odd length.

The proof of the Bernoulli identity \eqref{bpp} is straightforward.
Let
 $$ a_i = \left(\frac{2^{i+2} -2}{i+1}\right) B_{i+1}\ ,
\ \ \ A(x)=\sum_{i=0}^\infty a_i \frac{x^i}{i!} .$$
Using the definition of the Bernoulli numbers as
$$\frac{x}{e^x -1} = \sum_{i=0}^\infty B_i \frac{x^i}{i!}\ ,$$
we see
$$A(x) = \frac{2}{x} \sum_{i=0}^\infty (2^i-1) B_r
\frac{x^r}{r!} = \frac{2}{x}\left( \frac{2x}{e^{2x}-1} -
\frac{x}{e^x-1}\right)
 = -\left(\frac{2}{1+e^x}\right)\ .$$
The identity \eqref{bpp} follows from the series relation
$$e^x A(x)= - A(x)-2\ .$$

Formula \eqref{ggee} is valid 
for $\mathsf{R}(g,r,d,\sigma)$
 even when $\sigma$ has repeated parts:
the sum should be interpreted as running over all odd  
subsets $\tau\subset \sigma$ (viewing the parts of $\sigma$
as distinct).
\end{proof}

\subsection{Recasting}
We will recast the relations $\mathsf{R}(g,r,d,\sigma)$ in case
$g\equiv  r + |\sigma|+1$ mod 2 in a more convenient form. The result
will be crucial to the further analysis in Section \ref{trans}.

Let $g\equiv  r + |\sigma|+1$ mod 2, and  let
$\mathsf{S}(g,r,d,\sigma)$ denote the $\kappa$ polynomial
$$|\text{Aut}|\Big[ 
\exp\left(-\widetilde{\gamma}(t,x) +\sum_{\sigma \neq \emptyset} 
\Big( F_{\ell(\sigma),|\sigma|}+ \frac{\delta_{\ell(\sigma),1}}{2} \kappa_{|\sigma|-1}\Big) 
\frac{\mathbf{p}^\sigma}{|\text{Aut}(\sigma)|} \right) 
\Big]_{t^{r}x^d\mathbf{p}^\sigma}\ .$$
We can write $\mathsf{S}(g,r,d,\sigma)$ in terms of our previous
relations  $\mathsf{R}(g,r',d,\sigma')$ satisfying
$g\equiv r'+|\sigma'|+1$ mod 2 and $\sigma'\subset \sigma$:

If $g\equiv r+|\sigma|+1$ mod 2, then
for every subpartition $\tau\subset \sigma$ of even length
(including the case $\tau=\emptyset$), we have
$$g \equiv r- |\tau|+ \ell(\tau) + |\sigma/\tau| +1 \ \mod \ 2\ $$
where $\sigma/\tau$ is the complement of $\tau$.
The relation
$$\prod_i \kappa_{\tau_i-1} \cdot \mathsf{R}(g,r- |\tau|+ \ell(\tau), d,
\sigma/\tau)$$
is of codimension $r$.

In order to express $\mathsf{S}$ in terms of $\mathsf{R}$, we define
$z_i\in \mathbb{Q}$ by 
$$\frac{2}{e^x + e^{-x}} = \sum_{i=0}^\infty z_i \frac{x^i}{i!}\ .$$
Let $g\equiv r+|\sigma|+1$ mod 2, and let $\sigma$ have distinct parts. 
The formula
\begin{equation}\label{ggeee}
\mathsf{S}(g,r,d,\sigma)= \\ \sum_{\tau\subset \sigma}
\frac{z_{\ell(\tau)}}{2^{\ell(\tau)+1}} \cdot 
\prod_i \kappa_{\tau_i-1} \cdot \mathsf{R}\Big(g,r- |\tau|+ \ell(\tau), d,
\sigma/\tau\Big)
\end{equation}
follows again grouping like terms and 
the combinatorial identity
%\begin{eqnarray*}
%n\equiv 0\ \text{mod} \ 2: \ \ \ \sum_{k\geq 1} \binom{n}{2k-1} \left(\frac{2^{%2k+1} -2}{2k}\right) B_{2k}  & = & 0\ ,  \\
%n \equiv 1\ \text{mod} \ 2: \ \ \
%\sum_{k\geq 1} \binom{n}{2k-1} \left(\frac{2^{2k+1} -2}{2k}\right) B_{2k}  & = %& -
% \left(\frac{2^{n+2} -2}{n+1}\right) B_{n+1}\ .
%\end{eqnarray*}
\begin{equation} \label{bppp}
\sum_{i\geq 0} \binom{n}{i} \frac{z_i}{2^{i}+1}    = -
 \frac{z_n}{2^{n+1}} - \frac{1}{2^n}
\end{equation}
for $n>0$.
The sum in \eqref{ggeee} is over all subpartitions $\tau\subset \sigma$
of even length.

As before, there the identity \eqref{bppp} is straightforward to prove.
We see
 $$ 
\ \ \ Z(x)=\sum_{i=0}^\infty \frac{z_i}{2^{i+1}} \frac{x^i}{i!} =
\frac{1}{e^{x/2} + e^{-x/2}}\ .$$
The identity \eqref{bppp} follows from the series relation
$$e^x Z(x)= e^{x/2} -Z(x) .$$
 
Formula \eqref{ggee} is valid 
for $\mathsf{S}(g,r,d,\sigma)$
 even when $\sigma$ has repeated parts:
the sum should be interpreted as running over all even
subsets $\tau\subset \sigma$ (viewing the parts of $\sigma$
as distinct). We have proved the following result.

\begin{Proposition} In $R^r(\MM_g)$, the relation \label{better}
\begin{equation*}
\Big[ 
\exp\left(-\widetilde{\gamma}(t,x) +\sum_{\sigma \neq \emptyset} 
\Big( F_{\ell(\sigma),|\sigma|}+ \frac{\delta_{\ell(\sigma),1}}{2} \kappa_{|\sigma|-1}\Big) 
\frac{\mathbf{p}^\sigma}{|\text{\em Aut}(\sigma)|} \right) 
\Big]_{t^{r}x^d\mathbf{p}^\sigma}  = 0 
\end{equation*}
holds when 
$g-2d-1 +|\sigma| < r$ and $g\equiv  r + |\sigma|+1$ mod 2.
\end{Proposition}

\section{Transformation} \label{trans}
\subsection{Differential equations}

The function $\Phi$ satisfies a basic differential equation obtained
from the series definition:
$$\frac{d}{dx} (\Phi- tx \frac{d}{dx} \Phi) = -\frac{1}{t} \Phi \ .$$
After expanding and dividing by $\Phi$, we find
$$- tx \frac{\Phi_{xx}}{\Phi} - t \frac{\Phi_x}{\Phi}
+ \frac{\Phi_x}{\Phi} = -\frac{1}{t} \ $$
which can be written as
\begin{equation} \label{diffe}
-t^2x \gamma^*_{xx} = t^2 x (\gamma^*_x)^2 + t^2 \gamma^*_x 
-t \gamma^*_x -1 \ 
\end{equation}
where, as before, $\gamma^*=\log(\Phi)$.
Equation  \eqref{diffe} has been studied 
by Ionel in {\em Relations in the tautological ring} \cite{Ion}. We present
here  results of hers which will be useful for us.

To kill the pole and  match the required constant term,
 we will consider
the function
\begin{equation}\label{f4}
\Gamma=-t\left(\sum_{i\geq 1} \frac{B_{2i}}{2i(2i-1)}t^{2i-1} + \gamma^* 
\right)\ .
\end{equation}
%Via definition \eqref{}, we see
%\begin{equation}\label{f4}
%-\frac{1}{t} \Gamma(t,x) = \widetilde{gamma}(t,x)\ .
%\end{equation}
The differential equation \eqref{diffe} becomes
$$tx \Gamma_{xx} = x (\Gamma_x)^2 +(1-t) \Gamma_x -1 \ .$$
The differential equation is
easily seen to  uniquely determine $\Gamma$ once
  the initial conditions
$$\Gamma(t,0) = - 
\sum_{i\geq 1} \frac{B_{2i}}{2i(2i-1)} t^{2i}$$
are specified. 
By Ionel's first result,
$$\Gamma_x = \frac{-1+\sqrt{1+4x}}{2x} + \frac{t}{1+4x}
+ \sum_{k=1}^\infty \sum_{j=0}^k t^{k+1} q_{k,j}(-x)^j(1+4x)^{-j-\frac{k}{2}-1}\  $$
where the postive integers $q_{k,j}$ (defined to vanish unless
$k\geq j \geq 0$) are defined via the recursion
$$q_{k,j} = (2k+4j-2)q_{k-1,j-1} + (j+1)q_{k-1,j}
+ \sum_{m=0}^{k-1} \sum_{l=0}^{j-1} q_{m,l} q_{k-1-m,j-1-l}\ $$
from the initial value $q_{0,0}=1$.

Ionel's second result is obtained by integrating $\Gamma_x$
with respect to $x$. She finds
$$\Gamma = \Gamma(0,x) + \frac{t}{4} \log(1+4x)
-\sum_{k=1}^\infty \sum_{j=0}^k t^{k+1} c_{k,j} (-x)^j (1+4x)^{-j-\frac{k}{2}}\ $$
where the coefficients $c_{k,j}$ are determined by 
$$q_{k,j}=(2k+4j)c_{k,j} +(j+1) c_{k,j+1}$$
for $k\geq 1$ and $k\geq j\geq 0$.

While the derivation of the formula for $\Gamma_x$ is straightforward,
the formula for $\Gamma$ is quite subtle as the intial conditions
(given by the Bernoulli numbers) are used to show the
vanishing of constants of integration. Said differently, the
recursions for $q_{k,j}$ and $c_{k,j}$ must be shown to imply the formula
$$c_{k,0} = \frac{B_{k+1}}{k(k+1)}\ .$$
A third result of Ionel's is the determination of the
extremal $c_{k,k}$,
$$\sum_{k=1}^\infty c_{k,k} z^k = \log\left( \sum_{k=1}^\infty
\frac{(6k)!}{(2k)!(3k)!} \left(\frac{z}{72}\right)^k \right)\ .$$

The formula for $\Gamma$ becomes simpler after the following
very natural change of variables,
\begin{equation}\label{varch}
u= \frac{t}{\sqrt{1+4x}} \ \ \ \text{and} \ \ \
y= \frac{-x}{1+4x} \ .
\end{equation}
%Since $1+4y = (1+4x)^{-1}$, we see
%$$u = t (1+4y)^{\frac{1}{2}}\ .$$
The change of variables defines a new function
 $$\widehat{\Gamma}(u,y) = \Gamma(t,x) \ .$$
The formula for $\Gamma$ implies
\begin{equation}\label{f44}
\frac{1}{t} \widehat{\Gamma}(u,y) = \frac{1}{t} \widehat{\Gamma}(0,y)
-\frac{1}{4}\log(1+4y)
-\sum_{k=1}^\infty \sum_{j=0}^k  c_{k,j}u^k y^j  \ .
\end{equation}

Ionel's fourth result relates coefficients of series
after the change of variables \eqref{varch}.
Given any series
$$P(t,x) \in \mathbb{Q}[[t,x]],$$ let
$\widehat{P}(u,y)$ be the series obtained from
the change of variables \eqref{varch}. Ionel proves the
coefficient relation
$$\big[ P(t,x) \big]_{t^r x^d} = (-1)^d \big[
(1+4y)^{\frac{r+2d-2}{2}} \cdot \widehat{P}(u,y) \big]_{u^r y^d}\ .$$

\subsection{Analysis of the relations of Proposition \ref{vtw}}

We now study in detail the simple relations of Proposition \ref{vtw},
\begin{equation*}
\big[ \exp(-\widetilde{\gamma}) \big]_{t^rx^d} =0 \    \in R^r(\MM_g)
\end{equation*}
when $g-2d-1< r$ and 
$g\equiv r+1 \hspace{-5pt} \mod 2$.
Let 
$$\widehat{\gamma}(u,y) = \widetilde{\gamma}(t,x)$$
be obtained from the variable change \eqref{varch}.
Equations \eqref{f444}, \eqref{f4}, and \eqref{f44} together imply
\begin{equation*}
\widehat{\gamma}(u,y) = 
\frac{\kappa_0}{4}\log(1+4y)+
\sum_{k=1}^\infty \sum_{j=0}^k \kappa_k c_{k,j}u^k  y^j  \ 
\label{kakq}
\end{equation*}
modulo $\kappa_{-1}$ terms which we set to $0$.
 
Applying Ionel's coefficient result,
\begin{eqnarray*}
\big[ \exp(-\widetilde{\gamma}) \big]_{t^rx^d}& = & \big[ 
(1+4y)^{\frac{r+2d-2}{2}} \cdot
\exp(-\widehat{\gamma}) 
\big]_{u^r y^d} \\ & = &
\left[ 
(1+4y)^{\frac{r+2d-2}{2}-\frac{\kappa_0}{4}} \cdot
\exp(-     \sum_{k=1}^\infty \sum_{j=0}^k \kappa_k c_{k,j}u^k  y^j    ) 
\right]_{u^r y^d} \\
& = & 
\left[ 
(1+4y)^{\frac{r-g+2d-1}{2}} \cdot
\exp(-     \sum_{k=1}^\infty \sum_{j=0}^k \kappa_k c_{k,j}u^k y^j    ) 
\right]_{u^r y^d} \ .
\end{eqnarray*}
In the last line, the substitution $\kappa_0=2g-2$
has been made.

Consider first the exponent of $1+4y$.
By the assumptions on $g$ and $r$ in Proposition \ref{vtw},
$$\frac{r-g+2d-1}{2}\geq 0$$ 
and the fraction is integral. Hence, the $y$ degree of
the prefactor
$$(1+4y)^{\frac{r-g+2d-1}{2}}$$
is exactly $\frac{r-g+2d-1}{2}$.
The $y$ degree of the exponential factor is bounded
from above by the $u$ degree. We conclude
$$\left[ 
(1+4y)^{\frac{r-g+2d-1}{2}} \cdot
\exp(-     \sum_{k=1}^\infty \sum_{j=0}^k \kappa_k c_{k,j}u^k  y^j    ) 
\right]_{u^r y^d} =0$$
is the  {\em trivial} relation unless
$$r \geq d - {\frac{r-g+2d-1}{2}} = -\frac{r}{2} +\frac{g+1}{2} \ .$$
Rewriting the inequality, we obtain
$3r \geq g+1$
which is equivalent to $r > \lfloor \frac{g}{3} \rfloor$.
The conclusion is in agreement with the proven freeness of $R^*(\MM_g)$
up to (and including) degree $\lfloor \frac{g}{3} \rfloor$.

A similar connection between Proposition \ref{vtw} and Ionel's relations in
\cite{Ion} has also been found by Shengmao Zhu \cite{zhu}.

\subsection{Analysis of the relations of Theorem \ref{mmnn}}

For the relations of Theorem \ref{mmnn}, we will require additional
notation.  To start, 
let
$$\gamma^c(u,y) =      \sum_{k=1}^\infty \sum_{j=0}^k \kappa_k c_{k,j}u^k  y^j  \ .$$By Ionel's second result,
\begin{equation}\label{vvbb}
\frac{1}{t}\Gamma = \frac{1}{t} \Gamma(0,x) + \frac{1}{4} \log(1+4x)
-\sum_{k=1}^\infty \sum_{j=0}^k t^{k} c_{k,j} 
(-x)^j (1+4x)^{-j-\frac{k}{2}}\ .
\end{equation}
Let $c_{k,j}^0 = c_{k,j}$.
We  define the constants $c_{k,j}^n$ for $n\geq 1$ by
\begin{multline*}
\left( x \frac{d}{dx}\right)^n \frac{1}{t}\Gamma 
=\left( x \frac{d}{dx}\right)^{n-1} \left( \frac{-1}{2t} + \frac{1}{2t}\sqrt{1+4x}\right)\\ -
\sum_{k=0}^\infty \sum_{j=0}^{k+n} t^{k} 
c^n_{k,j} (-x)^j (1+4x)^{-j-\frac{k}{2}} 
\ . 
\end{multline*}

\begin{Lemma} \label{gqq2} For $n>0$, there are constants $b^n_j$ satisfying 
$$
\left( x \frac{d}{dx}\right)^{n-1} \left(\frac{1}{2t}\sqrt{1+4x}\right)
= \sum_{j=0}^{n-1} b^n_j u^{-1} y^j \ .
$$
Moreover, $b^n_{n-1} = -2^{n-2}\cdot(2n-5)!!$ where
$(-1)!!=1$ and $(-3)!!=-1$.
\end{Lemma}
\begin{proof}
The result is obtained by simple induction.
The negative evaluations $(-1)!!=1$ and $(-3)!!=-1$ arise from the
$\Gamma$-regularization. 
\end{proof}

\begin{Lemma}\label{ggrr} For $n>0$, we have $c_{0,n}^n = 4^{n-1}(n-1)!$.
\end{Lemma}

\begin{proof}
The coefficients $c_{0,n}^n$ are obtained directly from the $t^0$
summand $\frac{1}{4} \log(1+4x)$ of \eqref{vvbb}.
\end{proof}

\begin{Lemma} For $n>0$ and $k>0$, we have 
$$c_{k,k+n}^n =  (6k)(6k+4)\cdots (6k+4(n-1))\ 
c_{k,k}.$$
\label{gqq3}
\end{Lemma}

\begin{proof}
The coefficients $c^n_{k,k+n}$ are extremal. The differential
operators $x \frac{d}{dx}$ must always attack the $(1+4x)^{-j-\frac{k}{2}}$
in order to contribute $c^n_{k,k+n}$. The formula follows
by inspection.
\end{proof}

Consider next the full set of equations given by Theorem \ref{mmnn}
in the expanded form of Section \ref{LLL}.
The function $F_{n,m}$ may be rewritten as
\begin{eqnarray*}
F_{n,m}(t,x) & =& 
- \sum_{d=1}^\infty \sum_{s=-1}^\infty \widetilde{C}^d_{s}\ \kappa_{s+m} t^{s+m} 
\frac{d^n x^d}{d!} \\
& = & -t^m \left(x\frac{d}{dx}\right)^n 
\sum_{d=1}^\infty \sum_{s=-1}^\infty \widetilde{C}^d_{s}\ \kappa_{s+m} t^{s} 
\frac{x^d}{d!}.
\end{eqnarray*}
We may write the result in terms of the constants $b^n_j$ and $c^n_{k,j}$,
\begin{multline*}
t^{-(m-n)}F_{n,m} = -\delta_{n,1}\frac{\kappa_{m-1}}{2} \\ 
+
(1+4y)^{-\frac{n}{2}} \Big( \sum_{j=0}^{n-1}\kappa_{m-1}b_j^n u^{n-1} y^j
 -
\sum_{k=0}^\infty \sum_{j=0}^{k+n}  \kappa_{k+m}  c^n_{k,j} u^{k+n} y^j \Big)
\end{multline*}
Define the functions
$G_{n,m}(u,y)$ by
$$G_{n,m}(u,y) = \sum_{j=0}^{n-1}\kappa_{m-1}b_j^n u^{n-1} y^j
 -
\sum_{k=0}^\infty \sum_{j=0}^{k+n}  \kappa_{k+m}  c^n_{k,j} u^{k+n} y^j \ .$$

Let $\sigma=(1^{a_1}2^{a_2}3^{a_3} \ldots)$ be a partition
of length $\ell(\sigma)$ and size $|\sigma|$.
We assume the parity condition 
\begin{equation}\label{par2}
g\equiv  r + |\sigma|+1\ .
\end{equation}
Let $G_\sigma^\pm(u,y)$ be the following function associated to $\sigma$,
$$G_\sigma^\pm(u,y) =
\sum_{\sigma^{\bullet}\in \mathcal{S}(\sigma)}
\prod_{i=1}^{\ell(\sigma^{\bullet})}
\left(G_{\ell(\sigma^{(i)}), |\sigma^{(i)}|} \pm
\frac{\delta_{\ell(\sigma^{(i)}),1}}{2} 
\sqrt{1+4y}\ 
\kappa_{|\sigma^{(i)}|-1}\right)\ .$$
The relations of Theorem \ref{mmnn} in the
the expanded form of Section \ref{exf} written in the variables $u$ and $y$ are
\begin{equation*}
\Big[ (1+4y)^{\frac{r-|\sigma|-g+2d-1}{2}}
\exp(-\gamma^c) 
\left( G_\sigma^+ + G_\sigma^-\right)
\Big]_{u^{r-|\sigma|+\ell(\sigma)}y^d}  = 0 
\end{equation*}
In fact, the relations of Proposition \ref{better} here take a much more
efficient form. We obtain the following result.

%Let
%$$\Gamma^c = \gamma^c 
%%- \sum_{\sigma \neq \emptyset} G_{\ell(\sigma),|\sigma|} 
%\mathbf{p}^\sigma\ .$$

\begin{Proposition} In $R^r(\MM_g)$, the relation  \label{midb}
\begin{equation*}
\Big[ (1+4y)^{\frac{r-|\sigma|-g+2d-1}{2}}
\exp\left(-\gamma^c -\sum_{\sigma \neq \emptyset} G_{\ell(\sigma),|\sigma|} 
\frac{\mathbf{p}^\sigma}{|\text{\em Aut}(\sigma)|} \right) 
\Big]_{u^{r-|\sigma|+\ell(\sigma)}y^d\mathbf{p}^\sigma}  = 0 
\end{equation*}
holds when 
$g-2d-1 +|\sigma| < r$ and $g\equiv  r + |\sigma|+1$ mod 2.
\end{Proposition}

Consider the  exponent of $1+4y$.
By the inequality and the  parity condition \eqref{par2},
$$\frac{r-|\sigma|-g+2d-1}{2}\geq 0$$ 
and the fraction is integral. Hence, the $y$ degree of
the prefactor
\begin{equation}\label{xcx3}
(1+4y)^{\frac{r-|\sigma|-g+2d-1}{2}}
\end{equation}
is exactly $\frac{r-|\sigma|-g+2d-1}{2}$.
The $y$ degree of the exponential factor is bounded
from above by the $u$ degree. We conclude
the relation of Theorem 4 is
 {\em trivial} unless
$$r-|\sigma|+\ell(\sigma) \geq d - {\frac{r-|\sigma|-g+2d-1}{2}} 
= -\frac{r-|\sigma|}{2} +\frac{g+1}{2} \ .$$
Rewriting the inequality, we obtain
$$3r \geq g+1 + 3|\sigma|-2\ell(\sigma)$$
which is consistent with the proven freeness of $R^*(\MM_g)$
up to (and including) degree $\lfloor \frac{g}{3} \rfloor$.

\subsection{Another form}

A subset of the equations of Proposition \ref{midb} admits an especially simple
description.
Consider the function
\begin{multline*}
H_{n,m}(u) =  2^{n-2}(2n-5)!! \ \kappa_{m-1} u^{n-1} 
 + 4^{n-1}(n-1)!\ \kappa_m u^n \\
+ \sum_{k=1}^\infty  (6k)(6k+4)\cdots (6k+4(n-1)) c_{k,k}\   
\kappa_{k+m}
u^{k+n} \ .
\end{multline*}

\begin{Proposition} In $R^r(\MM_g)$, the relation \label{best}
\begin{equation*}
\Big[ 
\exp\left(-\sum_{k=1}^\infty c_{k,k} \kappa_k u^k 
-\sum_{\sigma \neq \emptyset} H_{\ell(\sigma),|\sigma|} 
\frac{\mathbf{p}^\sigma}{|\text{\em Aut}(\sigma)|} \right) 
\Big]_{u^{r-|\sigma|+\ell(\sigma)}\mathbf{p}^\sigma}  = 0 
\end{equation*}
holds when 
$3r \geq g+1 + 3|\sigma|-2\ell(\sigma)$
 and $g\equiv  r + |\sigma|+1$ mod 2.
\end{Proposition}

\begin{proof}
Let $g\equiv  r + |\sigma|+1$, and let
$$ \frac{3}{2}r -\frac{1}{2}g - \frac{1}{2}-\frac{3}{2} |\sigma| + \ell(\sigma) = \Delta >0\ .$$
By the parity condition, $\delta$ is an integer.
For $0 \leq \delta \leq \Delta$, let
$$\mathsf{E}_\delta(g,r,\sigma) = 
\Big[ 
%(1+4y)^{\frac{r-|\sigma|-g+2d-1}{2}}
\exp\left(-\gamma^c +\sum_{\sigma \neq \emptyset} G_{\ell(\sigma),|\sigma|} 
\frac{\mathbf{p}^\sigma}{|\text{Aut}(\sigma)|} \right) 
\Big]_{u^{r-|\sigma|+\ell(\sigma)}y^{r-|\sigma|+\ell(\sigma)-\delta}\mathbf{p}^\sigma} \ .
$$

The $\delta=0$ case is special.  
Only the monomials of $G_{n,m}$ of equal $u$ and $y$
degree contribute to the relations of Proposition \ref{midb}.
By Lemmas \ref{gqq2} - \ref{gqq3}, $H_{u,m}(uy)$ is exactly 
the subsum of $G_{n,m}$ consisting of such monomials.
Similarly,
$$\sum_{k=1}^\infty c_{k,k} \kappa_k u^ky^k $$
is the subsum of $\gamma^c$ of monomials of equal $u$ and $y$ degree.
Hence,
\begin{multline*}
\mathsf{E}_0(g,r,\sigma) = \\ 
\Big[ 
\exp\left(-\sum_{k=1}^\infty c_{k,k} \kappa_k u^ky^k 
-\sum_{\sigma \neq \emptyset} H_{\ell(\sigma),|\sigma|}(uy) 
\frac{\mathbf{p}^\sigma}{|\text{Aut}(\sigma)|} \right) 
\Big]_{(uy)^{r-|\sigma|+\ell(\sigma)}\mathbf{p}^\sigma} = \\
\Big[ 
\exp\left(-\sum_{k=1}^\infty c_{k,k} \kappa_k u^k 
-\sum_{\sigma \neq \emptyset} H_{\ell(\sigma),|\sigma|}(u) 
\frac{\mathbf{p}^\sigma}{|\text{Aut}(\sigma)|} \right) 
\Big]_{u^{r-|\sigma|+\ell(\sigma)}\mathbf{p}^\sigma} \ .
\end{multline*}

We consider the relations of Proposition \ref{midb} for
fixed $g$, $r$, and $\sigma$ as $d$ varies.
In order to satisfy the inequalty
$g-2d-1 +|\sigma| < r$, let
$$d(\widehat{\delta}) = \frac{-r+g +1+ |\sigma|}{2}+ \widehat{\delta}\ , \ \ \ \text{for} \ \ 
\widehat{\delta}\geq 0 .$$ 
For $0 \leq \widehat{\delta} \leq \Delta$, relation of
Proposition \ref{midb} for $g$, $r$, $\sigma$, and $d(\widehat{\delta})$
 is 
$$\sum_{i=0}^{\widehat{\delta}}  4^i \binom{\widehat{\delta}}{i}
\cdot \mathsf{E}_{\Delta-\widehat{\delta}+i}(g,r,\sigma) = 0\ .$$
As $\widehat{\delta}$ varies, we therefore obtain all the relations
\begin{equation}\label{lxx3}
\mathsf{E}_{\delta}(g,r,\sigma) = 0 \ 
\end{equation} 
for $0\leq \delta \leq \Delta$.
The relations of Proposition \ref{best} are obtained when $\delta=0$
in \eqref{lxx3}.
\end{proof}

The main advantage of Proposition \ref{best} is the dependence
on only the function
\begin{equation}\label{njjk}
\sum_{k=1}^\infty c_{k,k} z^k = \log\left( \sum_{k=1}^\infty
\frac{(6k)!}{(2k)!(3k)!} \left(\frac{z}{72}\right)^k \right)\ .
\end{equation}
Proposition \ref{best} only provides finitely many relations
for fixed $g$ and $r$.
In Section \ref{pppp},
we show Proposition \ref{best} is equivalent to the Faber-Zagier
conjecture. 

\subsection{Relations left behind}
In our analysis of relations obtained from the
virtual geometry of the moduli space of stable quotients, twice we have discarded
large sets of relations.
In Section \ref{exrel}, instead of analyzing all of the geometric
possibilities
\begin{equation*}
\nu_*\left(\prod_{i=1}^n\epsilon_*(s^{a_i} \omega^{b_i}) \cdot  0^c \cap [Q_g(\proj^1,d)]^{vir} \right)= 0\ \ 
\text{in} \  A^*(\MM_g,\mathbb{Q}) \ ,
\end{equation*}
we restricted ourselves to the case where $a_i=1$ for all $i$.
And just now, instead of keeping all the relations \eqref{lxx3}, we
restricted ourselves to the ${\delta=0}$ cases.

In both instances, the restricted set was chosen to allow further
analysis. In spite of the discarding, we will arrive at the
Faber-Zagier relations. We expect the discarded relations are all
redundant (consistent with Conjecture 2), but we do not have a proof.

\section{Equivalence}
\label{pppp}

\subsection{Notation}
The relations in Proposition~\ref{best} have a similar flavor to the Faber-Zagier relations. We start with formal series related  to
\[
A(z) = \sum_{i=0}^\infty \frac{(6i)!}{(3i)!(2i)!}\left(\frac{z}{72}\right)^i,
\]
we insert classes $\kappa_r$, we 
exponentiate, and we extract  coefficients to obtain 
relations among the $\kappa$ classes. 
In order to make the similarities clearer, 
we will introduce additional notation. 

If $F$ is a formal power series in $z$,
\[
F = \sum_{r=0}^\infty c_rz^r
\]
with coefficients in a ring, let
\[
\{F\}_\kappa = \sum_{r=0}^\infty c_r\kappa_rz^r
\]
be the series with $\kappa$-classes inserted.

Let $A$ be as above, and let
\[
B(z) = \sum_{i=0}^\infty \frac{(6i)!}{(3i)!(2i)!}\frac{6i+1}{6i-1}\left(\frac{z}{72}\right)^i
\]
be the second power series appearing in the Faber-Zagier relations.
Let
$$C = \frac{B}{A}\ ,$$ 
and let 
$$E = \exp(-\{\log(A)\}_\kappa) = \exp\left(-\sum_{k=1}^\infty c_{k,k}\kappa_kz^k\right).$$
We will rewrite the Faber-Zagier relations and the relations of Proposition~\ref{best} in terms of $C$ and $E$.
The equivalence between the two will rely on the principal differential equation
satisfied by $C$,
\begin{equation}\label{diffeq}
12z^2\frac{dC}{dz} = 1 + 4zC - C^2.
\end{equation}

\subsection{Rewriting the relations}
The relations conjectured by Faber and Zagier
are straightforward to rewrite using the above notation:
\begin{multline}\label{FZ0}
\Bigg[E\cdot\exp\Big(-\Big\{\log\big(1+p_3z+p_6z^2+\cdots\\
+C(p_1+p_4z+p_7z^2+\cdots)\big)\Big\}_\kappa\Big)\Bigg]_{z^rp^\sigma} = 0
\end{multline}
for $3r \ge g+|\sigma|+1$ and $3r\equiv g+|\sigma|+1$ mod $2$.
The above relation \eqref{FZ0} will be denoted $\FZ(r,\sigma)$.

The stable quotient relations of Proposition~\ref{best} are more complicated 
to rewrite in terms of $C$ and $E$. Define a sequence of power series $(C_n)_{n\ge 1}$ by
\begin{multline*}
2^{-n}C_n = 2^{n-2}(2n-5)!!z^{n-1} + 4^{n-1}(n-1)!z^n \\
+ \sum_{k=1}^\infty (6k)(6k+4)\cdots(6k+4(n-1))c_{k,k}z^{k+n}.
\end{multline*}
We see
$$H_{n,m}(z) = 2^{-n}z^{n-m}\{z^{m-n}C_n\}_\kappa.$$
 The series $C_n$ satisfy
\begin{equation}\label{Crecur}
C_1 = C, \ \ \ \ C_{i+1} = \left(12z^2\frac{d}{dz}-4iz\right)C_i.
\end{equation}
Using the differential equation \eqref{diffeq},
 each $C_n$ can be expressed as a polynomial in $C$ and $z$:
\[
C_1 = C, \ \ C_2 = 1-C^2,\ \  C_3 = -8z-2C+2C^3, \ldots, \ .
\]

Proposition~\ref{best} can then
be rewritten as follows (after an appropriate change of variables):
\begin{equation}\label{SQ}
\left[E\cdot\exp\left(-\sum_{\sigma\ne\emptyset}\{z^{|\sigma|-\ell(\sigma)}C_{\ell(\sigma)}\}_\kappa\frac{p^\sigma}{|\Aut(\sigma)|}\right)\right]_{z^rp^\sigma} = 0
\end{equation}
for $3r \ge g+3|\sigma|-2\ell(\sigma)+1$ and $3r \equiv g+3|\sigma|-2\ell(\sigma)+1$ mod $2$.
The above relation \eqref{SQ} will be denoted $\SQ(r,\sigma)$.

The $\FZ$ and $\SQ$
relations  now look much more similar, but the relations in (\ref{FZ0}) are 
indexed by partitions with no parts of size $2$ mod $3$ and
satisfy a slightly different inequality.
The indexing differences can be erased 
by observing that the variables $p_{3k}$ are actually not necessary in (\ref{FZ0}) 
if we are just interested in the \emph{ideal} generated by a set of relations 
(rather than the linear span).
This observation follows from the identity
\[
-\FZ(r,\sigma \sqcup 3a) = \kappa_a\FZ(r-a,\sigma) + \sum_\tau \FZ(r, \tau),
\]
where the sum runs over the $\ell(\sigma)$ partitions $\tau$ (possibly repeated) formed by increasing one of the parts of $\sigma$ by $3a$.

If we remove the variables $p_{3k}$
and reindex the others by replacing $p_{3k+1}$ with $p_{k+1}$, 
we obtain the following equivalent form of the $\FZ$ relations:
\begin{equation}\label{FZ}
\Big[E\cdot\exp\big(-\big\{\log(1+C(p_1+p_2z+p_3z^2+\cdots))\big\}_\kappa\big)\Big]_{z^rp^\sigma} = 0
\end{equation}
for $3r \ge g+3|\sigma|-2\ell(\sigma)+1$ and $3r \equiv g+3|\sigma|-2\ell(\sigma)+1$ mod $2$.

\subsection{Comparing the relations}

We now explain how to write the $\SQ$ relations (\ref{SQ}) as linear combinations of the 
$\FZ$ relations (\ref{FZ}) with coefficients in $\Q[\kappa_0,\kappa_1,\kappa_2,\ldots]$.
In fact,
the associated matrix will be triangular with diagonal entries equal to $1$.  

We start with further notation.
For a partition $\sigma$, let
\[
\FZ_\sigma = \left[\exp\left(-\left\{\log(1+C(p_1+p_2z+p_3z^2+\cdots))\right\}_\kappa\right)\right]_{p^\sigma}
\]
and
\[
\SQ_\sigma = \left[\exp\left(-\sum_{\sigma\ne\emptyset}\{z^{|\sigma|-\ell(\sigma)}C_{\ell(\sigma)}\}_\kappa\frac{p^\sigma}{|\Aut(\sigma)|}\right)\right]_{p^\sigma}
\]
be power series in $z$ with coefficients that are polynomials in the $\kappa$ classes. 
The relations themselves are given by 
$$\FZ(r,\sigma)=[E\cdot\FZ_\sigma]_{z^r}\ , \ \ \ \SQ(r,\sigma)=
[E\cdot\SQ_\sigma]_{z^r}\ .$$

It is straightforward to expand $\FZ_\sigma$ and $\SQ_\sigma$ as linear combinations of products of factors $\{z^a C^b\}$ for $a\ge 0$ and $b\ge 1$, with coefficients that are polynomials in the kappa classes.
When expanded, $\FZ_\sigma$ always contains exactly one term of the form 
\begin{equation}\label{g66h}
\{z^{a_1}C\}_\kappa\{z^{a_2}C\}_\kappa\cdots\{z^{a_m}C\}_\kappa\ . 
\end{equation}
All the other terms involve higher powers of $C$. 
If we expand $\SQ_\sigma$, 
we can look at the terms of the form \eqref{g66h} to determine what the coefficients must be 
when writing the $\SQ_\sigma$ as linear combinations of the $\FZ_\sigma$. For example,
\begin{align*}
\SQ_{(111)} &= -\frac{1}{6}\{C_3\}_\kappa + \frac{1}{2}\{C_2\}_\kappa\{C_1\}_\kappa -\frac{1}{6}\{C_1\}_\kappa^3 \\
&= \frac{4}{3}\kappa_1z + \frac{1}{3}\{C\}_\kappa - \frac{1}{3}\{C^3\}_\kappa + \frac{1}{2}(\kappa_0 - \{C^2\}_\kappa)\{C\}_\kappa - \frac{1}{6}\{C\}_\kappa^3 \\
&= \left(\frac{4}{3}\kappa_1z\right) + \left(\left(\frac{1}{3} + \frac{\kappa_0}{2}\right)\{C\}_\kappa\right)\\
& \ \ \ \ \ \ \ \ \ \ \ \  + \left(-\frac{1}{3}\{C^3\}_\kappa -\frac{1}{2}\{C^2\}_\kappa\{C\}_\kappa -\frac{1}{6}\{C\}_\kappa^3\right) \\
&= \frac{4}{3}\kappa_1z\FZ_\emptyset + \left(-\frac{1}{3} - \frac{\kappa_0}{2}\right)\FZ_{(1)} + \FZ_{(111)}.
\end{align*}

In general we must check that the terms involving higher powers of $C$ also match up.
The matching will require an identity between the coefficients of $C_i$ when expressed as polynomials in $C$.
Define polynomials $f_{ij}\in\Z[z]$ by
\[
C_i = \sum_{j=0}^if_{ij}C^j.
\]
It will also be convenient to write $f_{ij} = \sum_{k} f_{ijk}z^k$, so
\[
C_i = \sum_{\substack{ j,k\ge 0 \\ j + 3k \le i }} f_{ijk}z^kC^j.
\]
If we define
\[
F = 1+\sum_{i,j\ge 1}\frac{(-1)^{j-1}f_{ij}}{i!(j-1)!}x^iy^j \in \Q[z][[x,y]],
\]
then we will need a single property of the power series $F$.
\begin{Lemma}\label{exponential}
There exists a power series $G\in\Q[z][[x]]$ such that $F = e^{yG}$.
\end{Lemma}
\begin{proof}
The recurrence \eqref{Crecur} for the $C_i$ together with the differential equation \eqref{diffeq} satisfied by $C$ yield a recurrence relation for the polynomials $f_{ij}$:
\[
f_{i+1, j} = (j+1)f_{i, j+1} + 4(j-i)zf_{ij} - (j-1)f_{i, j-1}.
\]
This recurrence relation for the coefficients of $F$ is equivalent to a differential equation:
\[
F_x = -yF_{yy} + 4zyF_y - 4zxF_x + yF.
\]

Now, let $G\in\Q[z][[x,y]]$ be $\frac{1}{y}$ times the logarithm of $f$ (as a formal power series). The differential equation for $F$ can be rewritten in terms of $G$:
\[
G_x = -2G_y - yG_{yy}-(G + yG_y)^2 + 4z(G + yG_y) - 4zxG_x + 1.
\]
We now claim that the coefficient of $x^ky^l$ in $G$ is zero for all $k\ge 0, l\ge 1$, as desired. For $k = 0$ this is a consequence of the fact that $F = 1 + O(xy)$ and thus $G = O(x)$, and higher values of $k$ follow from induction using the differential equation above.
\end{proof}

We can now write the $\SQ_\sigma$ as linear combinations of the $\FZ_\sigma$.

\begin{Theorem}\label{combination}
Let $\sigma$ be a partition. Then $\SQ_\sigma - \FZ_\sigma$ is a $\Q$-linear combination of terms of the form
\[
\kappa_{\mu}z^{|\mu|}\FZ_{\tau},
\]
where $\mu$ and $\tau$ are partitions ($\mu$ possibly containing parts of size $0$) satisfying $\ell(\tau) < \ell(\sigma)$, $3|\mu| + 3|\tau| - 2\ell(\tau) \le 3|\sigma| - 2\ell(\sigma)$, and 
$$3|\mu| + 3|\tau| - 2\ell(\tau) \equiv 3|\sigma| - 2\ell(\sigma)\ \mod 2\ . $$
\end{Theorem}
\begin{proof}
We will need some additional notation for subpartitions. If $\sigma$ is a partition of length $\ell(\sigma)$ with parts $\sigma_1,\sigma_2,\ldots$ (ordered by size) and $S$ is a subset of $\{1,2,\ldots, \ell(\sigma)\}$, then let $\sigma_S \subset \sigma$ denote the subpartition consisting of the parts $(\sigma_i)_{i\in S}$.

Using this notation, we explicitly expand $\SQ_\sigma$ and $\FZ_\sigma$ as sums over set partitions of $\{1,\ldots, \ell(\sigma)\}$:
\[
\SQ_\sigma = \frac{1}{|\Aut(\sigma)|}\sum_{P\vdash \{1,\ldots,\ell(\sigma)\}}\prod_{S\in P} \left(\sum_{j,k}-f_{|S|,j,k}\{z^{|\sigma_S|-|S|+k}C^j\}_\kappa\right),
\]
\[
\FZ_\sigma = \frac{1}{|\Aut(\sigma)|}\sum_{P\vdash \{1,\ldots,\ell(\sigma)\}}\prod_{S\in P} \left((-1)^{|S|}(|S|-1)!\{z^{|\sigma_S|-|S|}C^{|S|}\}_\kappa\right).
\]
Matching coefficients for terms of the form \eqref{g66h} tells us what the linear combination must be. We claim 
\begin{align}
\label{dcczz}
\SQ_\sigma = &\sum_{\substack{R \vdash \{1,\ldots,\ell(\sigma)\} \\ P \sqcup Q = R \\ k:R\to \Z_{\ge 0}}}\frac{|\Aut(\sigma')|}{|\Aut(\sigma)|}\times \\ \nonumber
 &\prod_{S\in P}(-f_{|S|,0,k(S)}\kappa_{|\sigma_S|-|S|+k(S)}z^{|\sigma_S|-|S|+k(S)})\prod_{S\in Q}(f_{|S|,1,k(S)})\FZ_{\sigma'},
\end{align}
where $\sigma'$ is the partition with parts $|\sigma_S| - |S| + 1 + k(S)$ for $S\in Q$. 
Using the vanishing $f_{i,j,k} = 0$ unless $j + 3k \le i$ and $j + 3k \equiv i\mod 2$, we easily check the above expression for $\SQ_\sigma$
is of the desired type.

Expanding $\SQ_\sigma$ and $\FZ_{\sigma'}$ in \eqref{dcczz} and canceling out the terms involving the $f_{i,0,k}$ coefficients, it remains to prove 
\begin{align*}
&\sum_{\substack{Q\vdash \{1,\ldots,\ell(\sigma)\} \\ k:Q\to\Z_{\ge 0} \\ j:Q\to\mathbb{N}}}\prod_{S\in Q} \left(-f_{|S|,j(S),k(S)}\{z^{|\sigma_S|-|S|+k(S)}C^{j(S)}\}_\kappa\right) \\
&=
\sum_{\substack{Q\vdash \{1,\ldots,\ell(\sigma)\} \\ k:Q\to\Z_{\ge 0}}}\prod_{S\in Q}(f_{|S|,1,k(S)})\sum_{P\vdash \{1,\ldots,\ell(\sigma')\}}\prod_{S\in P}\left((-1)^{|S|}(|S|-1)!\{z^{|(\sigma')_S| - |S|}C^{|S|}\}_\kappa\right).
\end{align*}

A single term on the left side of the above equation is determined by choosing a set partition $Q_{\text{left}}$ of $\{1,\ldots,\ell(\sigma)\}$ and then for each part $S$ of $Q_{\text{left}}$ choosing a positive integer $j(S)$ and a nonnegative integer $k_{\text{left}}(S)$. We claim that this term is the sum of the terms of the right side given by choices $Q_{\text{right}}$, $k_{\text{right}}$, $P$ 
such that $Q_{\text{right}}$ is a refinement of $Q_{\text{left}}$ that breaks each part $S$ in $Q_{\text{left}}$ into exactly $j(S)$ parts in $Q_{\text{right}}$, $P$ is the associated grouping of the parts of $Q_{\text{right}}$, and the $k_{\text{right}}(S)$ satisfy 
$$k_{\text{left}}(S) = \sum_{T\subseteq S}k_{\text{right}}(T)\ .$$
 These terms all are integer multiples of the same product of $\{z^aC^b\}_\kappa$ factors, so we are left with the identity
\begin{equation}\label{l399}
\frac{(-1)^{j_0-1}}{(j_0-1)!}f_{i_0,j_0,k_0} = \sum_{\substack{P\vdash \{1,\ldots,i_0\} \\ |P| = j_0 \\ k:P\to\Z_{\ge 0} \\ |k| = k_0}}\prod_{S\in P}f_{|S|,1,k(S)}.
\end{equation}
to prove.

But by the exponential formula, identity \eqref{l399} is simply a
restatement of Lemma~\ref{exponential}.
\end{proof}

The conditions on the linear combination in Theorem~\ref{combination} are precisely those needed so that multiplying by $E$ and taking the coefficient of $z^r$ allows us to write any $\SQ$ relation as a linear combination of $\FZ$ relations. The associated matrix is triangular with respect to the partial ordering of partitions by size, 
and the diagonal entries are equal to $1$. Hence, the matrix
 is invertible. 
We conclude the $\SQ$ relations are equivalent to the $\FZ$ relations.

\vspace{+16 pt}
\noindent Departement Mathematik, ETH Z\"urich \\
\noindent rahul@math.ethz.ch

\vspace{+8 pt}
\noindent
Department of Mathematics, University of Michigan\\
pixton@umich.edu

\end{document}